\documentclass[12pt]{amsart}

\usepackage{amsgen, amsmath, amsfonts, amsthm, amssymb, enumerate,amscd, soul}
\usepackage[all]{xy}
\usepackage{cancel}
\usepackage{float}
\usepackage{xcolor}
\usepackage{changes}
\usepackage{mathtools}
\usepackage{booktabs}

\definecolor{blue}{rgb}{0,0,1}
\definecolor{red}{rgb}{1,0,0}
\definecolor{green}{rgb}{0,.6,.2}
\definecolor{purple}{rgb}{1,0,1}

\long\def\red#1\endred{\textcolor{red}{#1}}
\long\def\blue#1\endblue{\textcolor{blue}{#1}}
\long\def\purple#1\endpurple{\textcolor{purple}{ #1}}
\long\def\green#1\endgreen{\textcolor{green}{#1}}

\newtheorem{defn}{Definition}[section]
\newtheorem{prop}[defn]{Proposition}
\newtheorem{lem}[defn]{Lemma}
\newtheorem{thm}[defn]{Theorem}
\newtheorem{mainthm}{Theorem}

\newtheorem{conj}[defn]{Conjecture}

\newcommand\be{\begin{equation}}
\newcommand\ee{\end{equation}}

\newcommand {\sgn }{{\text{sgn}}}

\newcommand {\G}{{\Gamma}}

\newcommand {\R}{{\mathbb R}}

\newcommand {\g}{{\gamma}}

\newcommand {\ba}{{\mathfrak a}}
\newcommand {\bb}{{\mathfrak b}}

\newcommand {\ca}{{\mathbf a}}

\def\mod{\operatorname{mod}}

\def\Im{\operatorname{Im}}
\def\Re{\operatorname{Re}}

\newcommand{\supp}{\operatorname{supp}}

\def\ca{{\mathfrak a}}
\def\cb{{\mathfrak b}}

\def\sa{{\sigma_\mathfrak a}}
\def\sb{{\sigma_\mathfrak b}}

\begin{document}

\title{Sums of Kloosterman sums formed with modular symbols}

\author{Nikolaos Diamantis}
\address{School of Mathematical Sciences, The University of Nottingham, Nottingham, NG7 2RD, UK}
\email{nikolaos.diamantis@nottingham.ac.uk}

\author{Solomon Friedberg}
\address{Department of Mathematics, Boston College, Chestnut Hill, MA 02467, USA}
\email{solomon.friedberg@bc.edu}

\author{Fredrik Str\"omberg}
\address{School of Mathematical Sciences, The University of Nottingham, Nottingham, NG7 2RD, UK}
\email{fredrik.stromberg@nottingham.ac.uk}

\date{} 

\begin{abstract}
We study sums of Kloosterman sums formed with a modular symbol.  Employing Tauberian methods, we first give an estimate for a (Riesz) sum of Ramanujan sums formed with a modular symbol. We further define a zeta function that is analogous to the Kloosterman-Selberg zeta function, establish
its continuation to $\Re(s)>1/2$, give estimates for its growth and use this to prove a cancellation statement for sums of these twisted Kloosterman sums.
We explain the connection of this construction to the eigenvalue 1/4 problem and formulate an analogue of Linnik's conjecture.  Finally, we present numerical evidence
that there is cancellation and also that the Kloosterman sums with a modular symbol are not correlated with classical Kloosterman sums.
\end{abstract}

\thanks{This work was supported by the NSF, grant number DMS-2401309 (Friedberg) and by the Simons Foundation, Travel Support for Mathematicians award
MPS-TSM-00007379 (Friedberg). It was further partially supported by the EPSRC grant EP/V026321/1 (Str\"omberg).}

\subjclass[2020]{Primary 11L05; Secondary 11F11; 11F67; 11F72; 11L07.}
\keywords{Modular symbol, twisted Ramanujan sum, twisted Kloosterman sum, twisted Kloosterman-Selberg zeta function, exceptional eigenvalue.}

\maketitle

\section{Introduction}

A fundamental conjecture in automorphic forms states that each Maass cusp form on a congruence subgroup of $SL_2(\mathbb{Z})$ has eigenvalue $\lambda\ge1/4$.
As explained by Linnik \cite{L} and Selberg \cite{S}, this conjecture follows from the conjectural cancellation in sums of the Kloosterman sums
$$S(m,n;c)=\sum_{a\bmod c, (a,c)=1} \exp(2\pi i ({n}a+{m}\bar a)/c),\qquad a\bar a \equiv 1 \bmod c,$$
which appear in the Fourier coefficients of the Poincar\'e series.  (One also introduces more general sums with level and twisted by a Dirichlet
character.)  Indeed, it is expected that for fixed $m,n\ne0$,
$$\sum_{1\leq c<X} \frac{S(m,n;c)}{c} = \text{O}_{\epsilon,m,n}(X^\epsilon),$$
and this cancellation (and similar statements for the more general sums) would imply the eigenvalue $1/4$ conjecture.
Kuznetsov \cite{K}, Deshouillers and Iwaniec \cite{DI2}, and 
Goldfeld and Sarnak \cite{GS} showed that 
\begin{equation}\label{classicalKlsumcancellation}\sum_{1\leq c<X} \frac{S(m,n;c)}{c} = \text{O}(X^{1/6+\epsilon}).
\end{equation}
See also Goldfeld \cite{GL}
and Sarnak and Tsimerman \cite{ST}.  
To study this cancellation, for $\Re(s)$ sufficiently large one defines 
the Kloosterman-Selberg zeta function
$$Z_{m,n}(s)=\sum_{c=1}^\infty \frac{S(m,n;c)}{c^{2s}}.$$
The works cited above establish theorems about the meromorphic continuation and growth of this function (see for example, \cite{GS}, Theorem 1).

In a different direction, let $f$ be a holomorphic cusp form of weight two on a congruence subgroup $\Gamma_0(N)\subset SL_2(\mathbb{Z})$
and let $\langle f,\gamma \rangle$ be the modular symbol defined for $\gamma\in \Gamma_0(N)$ by
$$\langle f,\gamma \rangle= -2 \pi i\int_{i \infty}^{\gamma i \infty} f(w)\,dw.$$
The distribution of this symbol is of high interest; see Petridis and Risager \cite{PR}.
Goldfeld \cite{G} (and, in a different context, Bruggeman) introduced an Eisenstein series twisted by this modular symbol
$$E^*(z,s)=\sum_{\gamma\in\Gamma_\infty\backslash \Gamma_0(N)} \langle f,\gamma\rangle \, \text{Im}(\gamma z)^s,\qquad \Re(s)\gg0.$$
The study of this Eisenstein series was developed by Bruggeman, Diamantis, Goldfeld, O'Sullivan and Petridis \cite{BD, DO, G, O,P}.
One may also study  twisted Kloosterman sums (first introduced by O'Sullivan in \cite{OTh})
$$S^*(m,n;c)=\sum_{\substack{\gamma\in \Gamma_\infty \backslash \Gamma_0(N)/\Gamma_\infty\\ c_\gamma=c}} \langle f,\gamma\rangle \exp\left(2\pi i ({n}a_\gamma+{m}d_\gamma)/c_\gamma\right),\qquad \gamma=\begin{pmatrix}a_\gamma&b_\gamma\\c_\gamma&d_\gamma\end{pmatrix}.$$
The coefficients $S^*(m,0;c)$ (with $N|c$) appear in the Fourier coefficients of the Eisenstein series $E^*(z,s)$ and may be regarded as Ramanujan sums twisted
by a modular symbol.   One may also define these twisted Kloosterman sums for other cusps. 
There is a
Poincar\'e series defined using this modular symbol whose Fourier coefficients involve the general twisted Kloosterman sum $S^*(m,n;c)$.

The goal of this article is to analyze sums of twisted Kloosterman sums and the corresponding twisted Kloosterman-Selberg zeta function
$$Z^*_{m,n}(s)=\sum_{N|c} \frac{S^*(m,n;c)}{c^{2s}},\qquad \Re(s)\gg0.$$
Though it is natural to expect results for such sums that are broadly similar to the untwisted case, there are new complications that appear.
First, there are infinitely many poles of $E^*(z,s)$ on the critical line $\Re(s)=\tfrac12$.  Second, the modular symbols $\langle f, \gamma\rangle$
do not have modulus 1, but rather, are normally distributed in a suitable sense (Petridis and Risager \cite{PR}).  Third, the Poincar\'e series
whose Fourier coefficients involve the twisted Kloosterman sums $S^*(m,n;c)$ are not in the usual automorphic $L^2$ space and do not
appear amenable to spectral methods.  In fact, we realize $Z^*_{m,n}(s)$ as a Taylor coefficient of a real analytic function $Z^*_{m,n}(s,\epsilon)$ 
(see \eqref{Z*-Zchi}) and we must then study its behavior near $\Re(s)=\tfrac12$ uniformly in $\epsilon$.  

Our main results are as follows. Theorem~\ref{main} 
gives an estimate for the sum over $c$ of the twisted Ramanujan sums $S^*(n,0;c)$. (We restrict to prime level $N=p$ for convenience.)
To do so we use Tauberian methods, making use of estimates for $E^*(z,s)$ and its Fourier coefficients, keeping careful track of the archimedean contributions,
and employing the Mellin convolution theorem.  The estimates are precise enough to allow us to shift around the infinitely many poles
one at a time.  We obtain a family of estimates of the form
$$\sum_{c<\frac{\sqrt{\pi X |n|}}{p}} 
S^*(n, 0; pc) \psi_c(n, X)=\text{Main Term}+O(X^{1/2 + \epsilon})$$
where the weighting functions $\psi_c(n, X)$ are $\ll X^{{\frac12}+\epsilon}$, as $X\to\infty$.  Here the
Main Term could see exceptional eigenfunctions. See Theorem~\ref{main} for the precise result and Section~\ref{num-Taub}
for numerical investigations and a related conjecture.  

Our second main result concerns the twisted Kloosterman-Selberg zeta function $Z^*_{m,n}(s)$ with $m,n\neq0$.  We study this function
by means of pseudo-Laplacians and the Sobolev inequalities, an approach pioneered by Petridis \cite{P} following
the philosophy of Phillips-Sarnak \cite{PS}.  As noted above, this involves 
studying a family of perturbations.  We show (Theorem~\ref{meromcontin})
\begin{mainthm}
Suppose that $m,n\neq0$.
\begin{enumerate} 
\item The zeta function $Z^*_{m, n}(s)$ has meromorphic continuation to $\Re(s)>1/2$.
\item The only potential poles of $Z^*_{m, n}(s)$ in the strip $\Re(s) \in (1/2, 1)$ are those $s$ such that $s(1-s)$ is an exceptional eigenvalue. {Each such pole has order at most two.}
\item For $|t|>1$ and $\frac12<\sigma  < 1$, we have
$$Z_{m, n}^*(s) \ll_{m, n, N, f, \varepsilon} \frac{|t|^{\frac72+\epsilon}}{(\sigma-\frac12)^3}.$$
\end{enumerate}
\end{mainthm}
Then we establish (Theorem~\ref{Theorem2})
\begin{mainthm}\label{TheoremB} Suppose that $m,n\neq0$. Let
$$\beta\coloneqq \overline{\lim_{\substack{c \to \infty \\ N|c}}}\frac{\log|S^*(m, n; c)|}{\log c}.$$
{If $0<\beta<1$, there are $r_j,q_j\in\mathbb C$} and $\alpha_j \in (0, 1)$ such that, for all $\varepsilon>0$, 
$$\sum_{\substack{0<c<x \\ N|c}}\frac{S^*(m,n; c)}{c}=\sum_{j=1}^{M} {x^{\alpha_j}(q_j\log x+r_j)}+O_{m,n} \left (x^{\frac{7\beta}{9}+\varepsilon}\right ).$$
\end{mainthm}
\noindent
The summands in $j$ are related to possible exceptional eigenvalues. For more details, see the statement and proof of Theorem~\ref{Theorem2}.
This theorem shows a cancellation of the sums of twisted Kloosterman sums
that is broadly analogous to \eqref{classicalKlsumcancellation}.  Our experimental results (see Section \ref{1stconjecture}) 
suggest that $\beta$ is equal to $1/2$ in the cases we studied.

We expect that enough cancellation in the sums of twisted Kloosterman sums 
\begin{equation} \label{newsumcancel}
\sum_{\substack{0<c<x \\ N|c}}\frac{S^*(m,n; c)}{c}
\end{equation}
would imply the eigenvalue 1/4 conjecture.  We cannot prove
this in full, since the explicit formulas for the residues of $Z^*_{m,n}(s)$ coming from potential exceptional eigenvalues are given in terms
of expressions that are complicated and so not obviously nonzero.  In lieu of a formal proof, we investigate this question numerically.   
As explained above, the existence of exceptional eigenvalues could already affect the behavior of sums of Ramanujan sums twisted by modular
symbols.  We show that the data is consistent
with the full cancellation of such sums and
strongly suggests the non-existence of exceptional eigenvalues.  We formulate what is expected in
Conjecture~\ref{conj-ram-sums}.  This Conjecture implies the non-existence of exceptional eigenvalues.

We also numerically examine sums of the twisted Kloosterman sums $S^*(m,n;c)$ with $m,n\neq0$. We show that the data is consistent
with full cancellation of such sums, i.e.\ that \eqref{newsumcancel} is $\text{O}(x^\epsilon)$.
We formalize our expectations in Conjecture~\ref{new-SL-conj}, which asserts that
a Linnik-Selberg type cancellation holds.
It is also natural to wonder how any cancellation of the sums of twisted Kloosterman sums is related to the cancellation of the sums of classical
Kloosterman sums predicted by Linnik and Selberg.  We present data suggesting that the sums
$S(m,n;c)$ and $S^*(m,n;c)$ are not correlated.  This leads us to formulate
 Conjecture~\ref{no-correlation}, which quantifies the
lack of correlation between $S(m,n;c)$ and $S^*(m,n;c)$ in two separate ways.
 
The remainder of this paper is organized as follows.
In Section~\ref{Preliminaries} we give the general notation and recap the properties of the classical non-holomorphic Eisenstein series.  
In Section~\ref{Eis}
we introduce the Eisenstein series twisted by a modular symbol, give its Fourier expansion and bound its Fourier coefficients.  
Section~\ref{Tauberian-the-first} develops estimates for the sum of Ramanujan sums that involve the modular symbol
by using Tauberian methods exploiting the pole at $s=1$.
Section~\ref{Tauberian-the-second} continues with Tauberian methods but goes farther, taking into account not only the possible poles
due to exceptional eigenvalues but also the infinitely many poles on the line $\Re(s)=\tfrac12$.  This allows us
to establish our main theorem concerning sums of twisted Ramanujan sums, Theorem~\ref{main}.
In Section~\ref{6} we introduce Sobolev methods and use these to establish our two main results on the twisted Kloosterman-Selberg
zeta function $Z^*_{m,n}(s)$ and on sums of twisted Kloosterman sums,
Theorems~\ref{meromcontin} and \ref{Theorem2}.  The concluding section, Section~\ref{Cool-numerical-section}, gives an account of our numerical
investigations.  It also includes the formal statements of the conjectures concerning cancellation and concerning the lack-of-correlation.

\medskip

{\bf Acknowledgments.} We thank Jeff Hoffstein and Dan Nicks for helpful discussions and feedback on this paper. 

\section{Preliminaries}\label{Preliminaries}
\subsection{General notation} Fix a positive integer $N$ and let $\G=\G_0(N)$ denote the group of matrices
$\left ( \smallmatrix a & b \\ c & d \endsmallmatrix \right )
$
of determinant $1$ with $a, b, c, d \in \mathbb Z$ and $N|c$. We fix a set $\{\ba_i, i=1, \dots, m\}$ of inequivalent
cusps of $\G_{0}(N)$.  
We choose the $\ba$'s so that $\ba_1=\infty$ and
$\ba_{m}=0$.

For each $\ba$ we fix a scaling matrix $\sigma_{\ba}$ such that
$\sigma_{\ba}(\infty)=\ba$ and $\sigma_{\ba}^{-1} \Gamma_{\ba}
\sigma_{\ba}=\Gamma_{\infty}\coloneqq \{\pm T^n; n \in \mathbb Z\}$ (where $T=\left ( \begin{smallmatrix} 1 & 1 \\ 0 & 1 \end{smallmatrix} \right )$). In particular, we select
$\sigma_{\ba_1}=I, \,$
$\sigma_{\ba_{m}}=W_{N}$, where $I$ is the identity matrix and $W_{N}$
is the Fricke involution
$\left ( \smallmatrix 0 & -1/\sqrt{N} \\ \sqrt{N} & 0 \endsmallmatrix
\right ).$

Set  $X=\Gamma\backslash\mathbb H$. If $z=x+iy$, we let $ds^2=(dx^2+dy^2)/y^2$ and $d\mu(z)=dx\,dy/y^2$ be its hyperbolic metric and its hyperbolic area form respectively. We consider the space $L^2(X)=L^2(X,d\mu)$ with norm
$$    \|F\|_{L^2(X)}^2 =       \int_X |F(z)|^2\,d\mu(z).$$ The Laplace operator is $\Delta=-y^2(\partial_x^2+\partial_y^2).$

 We shall also adopt the notation that we may write $M$ in the form $M =
\left(\smallmatrix a_{_M}&b_{_M}\\c_{_M}&d_{_M}\endsmallmatrix\right)$ and
that $j(M, z)\coloneqq c_M z+d_M.$
Finally, for $k \in \mathbb Z,$ $g: \mathfrak{H} \to \mathbb{C}$ and $\gamma \in$ SL$_2(\mathbb{R})$ we set
$$(g|_{k}\gamma)(z)=g(\gamma z)j(\gamma, z)^{-k}.$$

Let now $$f(z)=\sum_{n=1}^{\infty}a(n) e^{2 \pi i n z}$$
be a cusp form of weight $r \in 2\mathbb N$.  
For $\Re(t)>(r+1)/2$ and $x \in \mathbb Q$, we define the additive twist of 
the L-function of $f$:
$$ L_f(t, x)=\sum_{n=1}^{\infty} \frac{a(n) e^{2 \pi i nx}}{n^t}.$$
and its ``completion"
\begin{equation} \label{Lcompl} \Lambda_f(t, x)\coloneqq (2 \pi)^{-t} 
\Gamma(t)L_f(t, x)=\int_{0}^{\infty}f(iy+x)y^{t-1}dy
\end{equation}
yielding the analytic continuation. When $x=0$, we will write $L_f(t)$ and $\Lambda_f(t)$ for the (un-twisted) $L$-function and its completion.

Let $r=2.$ In Prop. 3.8 of \cite{DHKL} it is proved that, for all $\varepsilon>0$
\begin{equation*}\label{boundtwisted1}\Lambda_f \left (t, \frac{a}{c} \right ) \ll_{\varepsilon, N}|c|^{\varepsilon} \quad \text{$\Re(t)=\frac{3}{2}+\varepsilon$, and} \, \, \, 
\Lambda_f \left (t, \frac{a}{c} \right )\ll_{\varepsilon, N} |c|^{1+\varepsilon} \quad \text{$\Re(t)=\frac{1}{2}-\varepsilon$}
\end{equation*}
for all $a/c$ with $(a, c)=1.$ By Phragm\'en-Lindel\"of we deduce that, for all $t$ with $\frac12-\varepsilon \le \Re(t) \le \frac{3}{2}+\varepsilon$,
\begin{equation}\label{boundtwisted}\Lambda_f \left (t, \frac{a}{c} \right ) \ll_{\varepsilon, N} 
|c|^{\frac{3}{2}-{ \Re(}t)+2\varepsilon}
\end{equation}
 
We define the modular symbol of a weight $2$ cusp form $f$ by
\begin{equation}\label{modsym}\langle f, \gamma \rangle\coloneqq  -2 \pi i\int_{i \infty}^{\gamma i \infty} f(w)dw=-2 \pi\Lambda_f(1, \gamma i \infty) \quad \text{for all $\gamma \in \Gamma_0(N).$}\end{equation}
We note that, for all $z \in  \mathbb H \cup \mathbb Q \cup \{i \infty\},$ we have
\begin{equation}\label{inde} -2 \pi i\int_{z}^{\gamma z} f(w)dw=\langle f, \gamma \rangle.\end{equation} A result by Eichler \cite{Ei} implies that
\begin{equation}\label{boundmods}\langle f, \gamma \rangle \ll_{N} \log(|c_{\gamma}|) \ll_{\varepsilon, N} |c_{\gamma}|^{\varepsilon} \quad \text{for all $\gamma \in \Gamma_0(N)$ with $c_{\gamma} \neq 0.$} \end{equation}

\subsection{Non-holomorphic Eisenstein series}
We recall the classical non-holomorphic Eisenstein series at the cusp $\ba$:
\begin{equation}
E_{\mathfrak a}(z, s)\coloneqq \sum_{\g \in \G_{\mathfrak a}\backslash
\G} \text{Im}(\sa^{-1} \gamma z)^s.
\label{es} \end{equation} 
This converges absolutely for $\Re(s)>1$ and has a meromorphic continuation to all $s\in \mathbb{C}$.  Its Fourier expansion has the form
\begin{equation}\label{FExpClas} 
E_{\mathfrak a}(\sb z, s)=\delta_{\ca \cb}y^s+\phi_{\ba \bb}(s)y^{1-s}+\sum_{n \neq 0} \phi_{\ba \bb}(n, s) W_s(nz)
\end{equation}
where \begin{equation}\label{Whit} W_s(nz)=2\sqrt{|n| \Im(z)} K_{s-\frac12}(2 \pi |n| \Im(z))e^{2 \pi i n x},
\end{equation} with the usual $K$-Bessel function $K_w(z)$ and $\delta_{\ca \cb}=1$, if $\ca=\cb$ and $0$ otherwise.  The coefficients $\phi_{\ba \bb}$ can be expressed in terms of the following Kloosterman sums. First set
$$C_{\ba \, \bb}=\left \{c>0; \left ( \smallmatrix * & * \\ c & * \endsmallmatrix \right ) \in \sa^{-1} \G \sb \right \}.$$ Then, for $c \in C_{\ba \, \bb},$ consider the classical Kloosterman sums
\begin{equation}\label{Kl} S_{\ba \, \bb}(m, n; c)\coloneqq \sum_{\substack{\g \in \Gamma_{\infty} \backslash \ \sa^{-1} \G \sb /\Gamma_{\infty} \\ \, \, c_\g=c}} e^{2 \pi i \left (n\frac{a_{\g}}{c}+m\frac{d_{\g}}{c} \right )}.
\end{equation}
Then,
\begin{equation} \label{0cl} \phi_{\ba \bb}(s)\coloneqq  \sqrt{\pi}\frac{\Gamma \left ( s-\frac12\right ) }{\Gamma(s)}\sum_{c \in C_{\ba \bb}}\frac{S_{\ba \bb}(0, 0; c)}{c^{2s}}
\, \, \text{and} \, \,  \phi_{\ba \bb}(n, s)\coloneqq \frac{\pi^s |n|^{s-1}}{\Gamma \left ( s\right ) }\sum_{c \in C_{\ba \bb}}\frac{S_{\ba \bb}(n, 0; c)}{c^{2s}}.
\end{equation}
We will write $E(z, s)$ for $E_{\infty}(z, s)$, $S(n, m; c)$ for $S_{\infty \infty}(n, m; c)$ and $\phi(s)$ (resp. $\phi(n, s)$) for $\phi_{\infty \infty}(s)$ (resp. $\phi_{\infty \infty}(n, s)$).

 We let $\Phi(s)\coloneqq (\phi_{\ca_i \ca_j}(s))_{i, j=1}^m$ be the scattering matrix of $\mathbb E(z, s)$. Set $$\mathbb E(z, s)=(E_{\ca_1}(z, s), \dots, E_{\ca_m}(z, s))^T,$$
where $T$ denotes matrix transpose. Then we have the functional equation
$$\mathbb E(z, s)=\Phi({s})\mathbb E(z, 1-s).$$

To compute the Fourier coefficients of $E(z, s)$ in some special cases that we will need, we restrict to prime level $N=p$ for simplicity. Recall that, $\ba_1=\infty$, $\ba_2=0$, $\sigma_{\ba_1}=I$ and 
$\sigma_{\ba_2}=W_p.$ 

According to Lemma 4.6 in pg. 535 of \cite{Hej}, 
\begin{equation}\label{scat}\Phi(s)=\sqrt{\pi}\frac{\Gamma(s-\frac12)}{\Gamma(s)}\frac{\zeta(2s-1)}{\zeta(2s)}
\left ( \begin{matrix} \frac{p-1}{p^{2s}-1} &  \frac{p^s-p^{1-s}}{p^{2s}-1} \\ \frac{p^s-p^{1-s}}{p^{2s}-1} & \frac{p-1}{p^{2s}-1} \end{matrix} \right )=\frac{\zeta^*(2-2s)}{\zeta^*(2s)}
\left ( \begin{matrix} \frac{p-1}{p^{2s}-1} &  \frac{p^s-p^{1-s}}{p^{2s}-1} \\ \frac{p^s-p^{1-s}}{p^{2s}-1} & \frac{p-1}{p^{2s}-1} \end{matrix} \right )\end{equation}
where $\zeta^*(s)=\pi^{-s/2}\Gamma(s/2)\zeta(s)$.

An explicit description of $C_{\ba \, \bb}$ and $\Gamma_{\infty} \backslash \ \sa^{-1} \G \sb /\Gamma_{\infty}$ will also be useful. We have
\begin{equation}\label{C}C_{\infty \infty}=C_{00}=p \mathbb N, C_{\infty 0}=C_{0\infty}=\{\sqrt{p} \, n; n \in \mathbb N; (n, p)=1 \}\end{equation}
and a set of representatives of $\Gamma_{\infty} \backslash \ \sa^{-1} \G \sb /\Gamma_{\infty}$ is given by
\begin{align}\label{doublecoset}
\left \{ \left ( \begin{smallmatrix} a & * \\ pc & d \end{smallmatrix}\right ); c>0, 0\le d, a <pc; ad \equiv 1 \mod pc \right \} \, \, &\text{if $\ba=\bb=\infty$ or $\ba=\bb=0$} \nonumber\\
\left \{ \left ( \begin{smallmatrix} c\sqrt{p} & * \\ a\sqrt{p} & b\sqrt{p} \end{smallmatrix}\right ); a>0; 0\le b, c<a; pbc \equiv 1 \mod a \right \} \, \, &\text{if $\ba=\infty, \bb=0$ or $\ba=0, \bb=\infty$} 
\end{align}
In each of these descriptions the starred entry is uniquely determined so that the matrix has
determinant one.

\section{Eisenstein series modified with modular symbols}\label{Eis}
To each cusp form $f$ of weight 
$2$ for $\G=\Gamma_0(N)$ we attach the (weight $0$) non-holomorphic 
Eisenstein series modified with modular symbols at the cusp $\ba$ as follows:
\begin{equation}
E^*_{\mathfrak a}(z, s)\coloneqq \sum_{\g \in \G_{\mathfrak a}\backslash
\G}\langle f, \gamma \rangle \text{Im}(\sa^{-1} \gamma 
z)^s.
\label{esms2v} \end{equation} 
This converges absolutely for $\Re(s)>1$ \cite[Prop. 2.6]{PR} and its meromorphic continuation and functional equations have been established by Goldfeld \cite{G}, O'Sullivan \cite{O}, Petridis \cite{P} etc. In \cite{O} the following Fourier expansion of $E_{\ba}^*$ is given:
For $c \in C_{\ba \, \bb},$ we consider the modified Kloosterman sums
\begin{equation}\label{Kl*} S^*_{\ba \, \bb}(m, n; c)\coloneqq \sum_{\substack{\g \in \Gamma_{\infty} \backslash \ \sa^{-1} \G \sb /\Gamma_{\infty} \\ \, \, c_\g=c}}\langle f, \sa \g \sb^{-1} \rangle e^{2 \pi i \left (n\frac{a_{\g}}{c}+m\frac{d_{\g}}{c} \right )}.
\end{equation}
Then, for $\Re(s)>1$ we have 
\begin{equation}\label{FExp} 
E^*_{\mathfrak a}(\sb z, s)=\phi_{\ba \bb}^*(s)y^{1-s}+\sum_{n \neq 0} \phi_{\ba \bb}^*(n, s) W_s(nz)
\end{equation}
where $W_s(nz)$ is defined in \eqref{Whit} and 
\begin{align} \label{0} \phi_{\ba \bb}^*(s)&\coloneqq  \sqrt{\pi}\frac{\Gamma \left ( s-\frac12\right ) }{\Gamma(s)}\sum_{c \in C_{\ba \bb}}c^{-2s}S_{\ba \bb}^*(0, 0; c)
\\
\label{n} \phi_{\ba \bb}^*(n, s)&\coloneqq \frac{\pi^s |n|^{s-1}}{\Gamma \left ( s\right ) }\sum_{c \in C_{\ba \bb}}c^{-2s}S_{\ba \bb}^*(n, 0; c).
\end{align}
In particular, this shows that the Fourier coefficients of $E^*_{\ba}$ are Dirichlet series. 

With \eqref{inde} and \eqref{modsym}, we can relate the modified Kloosterman sums to values of twisted L-series: 
\begin{multline} -(2 \pi i)^{-1} \langle f, \sa \g \sb^{-1} \rangle=\int_{\sb i \infty}^{\sa \gamma \sb^{-1}(\sb i \infty)}f(w)dw=\left (\int_{\sa i \infty}^{\sa \gamma \sb^{-1}(\sb i \infty)}+\int_{\sb i \infty}^{\sa  i \infty} \right ) f(w)dw\\
=\int_{i \infty}^{\gamma i \infty}f|_2\sa(w)dw+\int_{\bb}^{\ba}f(w)dw=-i\Lambda_{f|_2\sa}(1; \gamma i\infty)+\int_{\bb}^{\ba}f(w)dw
\end{multline}
and hence, for $c \in C_{\ba \, \bb},$ we have
\begin{equation}\label{KlooTwist}S^*_{\ba \, \bb}(n, 0; c)= -2 \pi \left (\sum_{\substack{\g \in \Gamma_{\infty} \backslash \ \sa^{-1} \G \sb /\Gamma_{\infty} \\ \, \, c_\g=c}} \Lambda_{f|_2\sa}(1; \gamma i\infty) e^{\frac{2 \pi i n d_{\g}}{c}} \right ) -2 \pi i S_{\ba \, \bb}(n, 0; c)\int_{\bb}^{\ba}f(w)dw.
\end{equation}

With these notations we can formulate the functional equation of the Eisenstein series with modular symbols. Set
$\mathbb E^*(z, s)=(E^*_{\ca_1}(z, s), \dots, E^*_{\ca_m}(z, s))^T.$ Then, in \cite{O} it is shown that
\begin{equation}\label{FE}
\Phi(s)\mathbb E^*(z, 1-s)=\mathbb E^*(z, s)-\Phi^*(s)\Phi(1-s)\mathbb E(z, s)=\mathbb E^*(z, s)-\Phi^*(s)\mathbb E(z, 1-s)
\end{equation}
where $\Phi^*(s)=(\phi^*_{\ca_i \ca_j}(s))_{i, j=1}^m$ is its analogue of the scattering matrix for $\mathbb E^*(z, s)$. 

We will write $E^*(z, s)$ for $E^*_{\infty}(z, s)$, $S^*(n, m; c)$ for $S^*_{\infty \infty}(n, m; c)$ and $\phi^*(s)$ (resp. $\phi^*(n, s)$) for $\phi^*_{\infty \infty}(s)$ (resp. $\phi^*_{\infty \infty}(n, s)$).

 \subsection{Basic bounds of Fourier coefficients of $E^*(z, s)$} From now on until the end of the paper, we assume that $f$ is a Hecke eigenform normalised so that $a(1)=1$. We will also assume that the level $N$ will equal a prime $p$.
We will establish some bounds for the Fourier coefficients of $E^*(z, s)$ for various ranges of $\Re(s).$

\subsubsection{The region $\Re(s) \ge \frac12$} First, in the region where $E(z, s)$ and $E^*(z, s)$ converge absolutely as series, we have the following bounds.
Let $s=\sigma+it$ with $\sigma=1+\varepsilon$. Then, for $\ba=\infty$ and $0$, we note, with \eqref{doublecoset} and \eqref{boundmods}, that 
\begin{equation}\label{boundS*}S_{\ba \infty}(m, n; c) \, \, \text{and}\, \, S^*_{\ba \infty}(m, n; c) \ll_{p, \varepsilon} c^{\varepsilon} \cdot c=c^{1+\varepsilon}\end{equation}
Then from \eqref{n} and the formula for $\phi_{\ba \infty}(n, s)$, we deduce that
\begin{equation}\label{bound2plus} \Gamma(s) \phi_{\ba \infty}(n, s) \, \, \text{and}\, \, \Gamma(s)\phi_{\ba \infty}^*(n, s) \ll_{\varepsilon, n} \sum_{c \in C_{\ba \infty}}c^{-2-2\varepsilon}c^{1+\varepsilon} \ll 1 \quad \text{for $\sigma=1+\varepsilon$}.
\end{equation}

In the region $1/2+\varepsilon<\sigma \le 1$, where the series defining $E_{\ba}$ and $E^*_{\ba}$ do not converge absolutely, we have the following bounds proved in \cite[Theorem 3.3]{PR} and \cite[Lemma 3.1]{PR} respectively. 
\begin{equation}\label{boundE*E}E_{\ba}^*(z, s)=O_{f, \varepsilon, \ba}(|t|^{10(1-\sigma)+\varepsilon}) \quad \text{and $E_{\ba}(z, s)=O_{\varepsilon, \ba}(|t|^{4(1-\sigma)+\varepsilon}),$}\qquad \text{as $|t| \to \infty$}\end{equation}
uniformly for $z$ in a compact set $K.$  
As pointed out in \cite{PR}, although the bound is proved there for groups with one cusp only, the statement holds in general. 

To deduce a bound for $\phi^*_{\ba \infty}(n, s)$  from \eqref{boundE*E}, we will use the following bounds which are surely known but we provide a proof here because we were not able to find a precise reference for them. 
\begin{lem}\label{boundK} {Let} $x >0$ and $\nu=\mu+it$. Then we have:
$$I_\nu(x)=\frac{\left(\frac x2\right)^\nu}{\Gamma(\nu+1)} \left (1+O\!\left(\frac1{|t|}\right) \right ) \, \,  \text{and, if $\mu>0$, $K_\nu(x)=\frac{\Gamma(\nu)}{2} \left(\frac{x}{2}\right)^{-\nu} \left (1+O\!\left(\frac1{|t|^{\min(1, 2 \mu)}}\right) \right )$} $$
uniformly for $\mu \in [a, b]$ for $a, b \in \mathbb R.$ If $\mu=0,$ we have $K_\nu(x)=O_x \left (e^{\frac{-\pi |t|}{2}}|t|^{-\frac12}\right )$
\end{lem}
\begin{proof} From the series definition of $I_{\nu}(x)$ ($x>0$) we have 
\begin{equation}\label{I1}
I_\nu(x)=\sum_{k=0}^{\infty}
\frac{\left(\frac x2\right)^{2k+\nu}}{k!\,\Gamma(k+\nu+1)}
=
\frac{\left(\frac x2\right)^\nu}{\Gamma(\nu+1)} (1+R), \, \,  \text{where $R\coloneqq 
\sum_{k=1}^{\infty}
\frac{\Gamma(\nu+1)}{k!\Gamma(k+\nu+1)}
\left(\frac{x^2}{4}\right)^k.$}
 \end{equation} 
Since $|\Gamma(k+\nu+1)/\Gamma(\nu+1)|=|(\nu+1) \dots (\nu+k)| \ge |t|^k$, we deduce that 
\begin{equation}\label{I2}
|R|
\le
\sum_{k=1}^{\infty} \frac1{k!}
\left(\frac{x^2}{4|t|}\right)^k
=
e^{x^2/(4|t|)}-1
=
O\!\left(\frac1{|t|}\right), \quad \text{uniformly for $\mu \in [a,b]$.} \end{equation} For the asymptotics of $K_{\nu}(x)$ we 
combine \eqref{I1} and \eqref{I2} with the identity, valid for $\nu\notin\mathbb Z$,
$$
K_\nu(x)=\frac{\pi}{2\sin(\pi\nu)}
\left(I_{-\nu}(x)-I_\nu(x)\right)=\frac{1}{2}\Gamma(\nu)\Gamma(1-\nu)\left(I_{-\nu}(x)-I_\nu(x)\right).
$$
Since, by Stirling, $|\Gamma(1-\nu)/\Gamma(1+\nu)| \asymp |t|^{-2\mu}$, we deduce that, if $\mu>0$, 
\begin{multline*}K_\nu(x)=\frac{\Gamma(\nu)}{2}\left(\frac{x}{2}\right)^{-\nu} \left (1+O\!\left(\frac1{|t|}\right)- \frac{\Gamma(1-\nu)}{\Gamma(1+\nu)}\left(\frac{x}{2}\right)^{2\nu}\left ( {1}+O\!\left(\frac1{|t|}\right) \right )  \right)\\
=\frac{\Gamma(\nu)}{2}\left(\frac{x}{2}\right)^{-\nu} \left (1+O\!\left(\frac1{|t|^{\min(1, 2\mu)}}\right) \right), \quad \text{uniformly for $\mu \in [a,b] \subset \mathbb R_+$.}
\end{multline*}
The asymptotics in the case $\mu=0$, follows directly from \cite[Proposition 2]{BFT}.
\end{proof}
We can now prove the following
\begin{prop} For a cusp $\ba$, a fixed $n \neq 0$ and $s=\sigma+it$ with $\frac12<\sigma \le 1$, we have
\begin{equation}\label{boundphi}\Gamma(s) \phi^*_{\ba \infty}(n, s) \ll |t|^{\frac{21}{2}-10\sigma+\varepsilon}
 \, \, \, \text{and  $\phi^*_{\ba \infty}(s) \ll |t|^{10(1-\sigma)+\varepsilon}$}
\end{equation}
as $|t| \to \infty$, uniformly for $\sigma \in [a, b] \subset \left ( \frac12, 1 \right ].$ \end{prop}
\begin{proof}
We combine Lemma \ref{boundK} for $x=2 \pi |n|$ and $\nu=s-\frac12$ with \cite[5.11.12]{DLMF} to obtain
\begin{equation}\label{boundK1}{ \Big |}\frac{\Gamma(s)}{K_{s-\frac12}(2 \pi |n|)}{\Big |}={\Big |}\frac{\Gamma \left (s-\frac12 \right )}{K_{s-\frac12}(2 \pi |n|)} \frac{\Gamma \left (s \right )}{\Gamma \left (s-\frac12 \right )} { \Big |}\sim 
2\left( \pi |n|\right)^{ \Re (s)-\frac12}\sqrt{|t|} \, \, \text{as $|t| \to \infty,$}\end{equation} uniformly  for $\sigma \in [a, b] \subset \left ( \frac12, 1 \right ].$  We note that, since the only zeros in $w$ of $K_w(x)$ for a fixed $x>0$ lie on the imaginary axis (\cite{Ga}), the above ratio is well-defined for $\Re(s)>\frac{1}{2}$. 
On the other hand, \eqref{FExp} implies
$$\Gamma(s) \phi^*_{\ba \infty}(n, s)=\frac{\Gamma(s)}{2\sqrt{|n|}K_{s-\frac12}(2 \pi |n|)}\int_0^1E^*_{\ba}(x+i, s)e^{-2 \pi i n x}dx \, \, \text{and $\phi^*_{\ba \infty}(s)=\int_0^1E^*_{\ba}(x+i, s)dx$}.$$
Applying \eqref{boundK1} and \eqref{boundE*E} to this we deduce the assertion. 
\end{proof}

\subsubsection{The region $\Re(s) < \frac12$}
To study the growth of $\Gamma(s)\phi_{\ba \infty}^*(n, s)$ to the left of the line $\Re(s)=\frac12$ we will use the functional equation \eqref{FE}. We first note that, by Prop. 4.2 of \cite{O}, we have that $\phi_{\infty \infty}^*(s)=\phi_{0 0}^*(s)=0$ and $\phi_{\infty 0}^*(s)=-\phi_{0 \infty}^*(s)$. Combining this with \eqref{scat}, \eqref{FExp} and \eqref{FE}, we obtain
\begin{multline}\label{Fcoeff} \Gamma(s)\phi_{\infty \infty}^*(n, s)=\frac{\Gamma(s) \zeta^*(2-2s)}{(p^{2s}-1)\zeta^*(1-2s)} \left ( (p-1)\phi_{\infty \infty}^*(n, 1-s)+(p^s-p^{1-s})\phi_{0 \infty}^*(n, 1-s) \right ) \\+\Gamma(s) \phi_{\infty 0}^*(s) \phi_{0 \infty}(n, 1-s)
\end{multline}
We next observe that we can compute $\phi_{\infty 0}^*(s)$ quite explicitly. 
\begin{lem}\label{phi*infty0} 
The coefficient 
of $y^{1-s}$ in the Fourier expansion of $E^*_{\infty}(z, s)$ at the cusp $0$ equals 
\begin{equation}\label{phi*0}\phi_{\infty 0}^*(s)={-}\frac{\sqrt{\pi}\Gamma(s-\frac12){L_f(1)L_f(2s)}}{\Gamma(s) \zeta(2s)^2 p^s(1-p^{-2s})}
+\frac{\pi^{2s-1}(1-p^{1-2s})\Gamma (1-s)\zeta(2-2s)L_{f}(1)}{p^{s}\Gamma(s) \zeta(2s)(1-p^{-2s})}
\end{equation}
This implies that, for $\Re(s)=\frac12-\varepsilon,$ 
 \begin{equation}\label{boundphi*0}\zeta(2s)^2 \phi_{\infty 0}^*(s) \ll_{f, p, \varepsilon} |\Im(s)|^{\varepsilon} \qquad \text{as $|\Im(s)| \to \infty.$}\end{equation}
\end{lem}
\begin{proof} We prove \eqref{phi*0} for $\Re(s) \gg 0$ and then obtain the result for all $s$ by continuation. Suppose that $\Re(s) \gg 0.$ {Since $f$ is a normalised Hecke eigenform and the level is prime we have $f|_2W_p=w_p f$, with $w_p=-a(p) \in\{1,-1\}$.}
With \eqref{n} and \eqref{KlooTwist}, we have
\begin{multline}\label{phi0infty}
\phi^*_{0 \infty}(s)
=2 \pi i \sqrt{\pi}\frac{\Gamma \left ( s-\frac12\right ) }{\Gamma(s)} \\
\times \sum_{c \in C_{0 \infty}}c^{-2s}\left (i\sum_{\substack{\g \in \Gamma_{\infty} \backslash W_p^{-1} \Gamma /\Gamma_{\infty} \\ \, \, c_\g=c}}{w_p\Lambda_f}(1; \gamma i\infty)+
\left ( \int_{0}^{\infty}f(w)dw \right ) S_{0 \, \infty}(0, 0; c) \right ).
\end{multline}
The sum over the double coset equals the value at $t=1$ of the analytic continuation of
$$\frac{\Gamma(t)}{(2\pi)^t}\sum_{n>0}\frac{{w_pa(n)}}{n^t}\sum_{\substack{\g \in \Gamma_{\infty} \backslash \ W_p^{-1} \G /\Gamma_{\infty} \\ \, \, c_\g=c}} e^{2 \pi i \left (n\frac{a_{\g}}{c} \right )}$$
From the expression for the Fourier coefficients $\phi_{0, \infty}(n, s), \phi_{0 \infty}(s)$ of the standard Eisenstein series $E_{0}(z, s)$ we deduce from \eqref{phi0infty} that
$\phi^*_{0\infty}(s)$ equals the value at $t=1$ of the analytic continuation of
\begin{equation}\label{phi0infty1} -2 \pi\sqrt{\pi}\frac{\Gamma(t)}{(2\pi)^t}\pi^{-s}\Gamma \left (s-\frac12 \right )\sum_{n>0}\frac{{w_pa(n)}}{n^{t+s-1}}\phi_{0 \infty}(n, s)-2 \pi \Lambda_f(1)\phi_{0 \infty}(s).\end{equation}
To compute $\phi_{0 \infty}(n, s), \phi_{0 \infty}(s)$, we notice that, by (3.25) of \cite{CoIw}, we have
$$E_0(z, s)=p^{-s}(1-p^{-2s})^{-1}\left (G(z,s)-p^{-s}G(pz, s) \right )$$
where $$G(z, s)=\frac{y^s}{2}\sum_{(c, d)=1}|cz+d|^{-2s}$$
denotes the non-holomorphic Eisenstein series for the full modular group. Therefore, with the well-known formulas for the Fourier coefficients of $G(z, s)$ we obtain, for $n \neq 0$,
$$\phi_{0 \infty}(n, s)=\frac{\pi ^s p^{-s}}{|n|^s\Gamma(s)\zeta(2s)(1-p^{-2s})}\left (\sigma_{2s-1}(|n|)-\sigma_{2s-1}(|n|/p) \right ),
$$ where $\sigma_{2s-1}(n)$ is taken to be $0$, if $n \not \in \mathbb Z$, and, with the functional equation of $\zeta(s),$
$$\phi_{0 \infty}(s)=\frac{\pi^{2s-1}(1-p^{1-2s})\Gamma (1-s)\zeta(2-2s)}{p^{s}\Gamma(s) \zeta(2s)(1-p^{-2s})}.$$
Then, the first term of \eqref{phi0infty1} becomes
$$-\frac{\Gamma(t)}{(2\pi)^{t-1}}\frac{ \sqrt{\pi} \Gamma(s-\frac12)}{p^s(1-p^{-2s})\Gamma(s)\zeta(2s)}\sum_{n>0}\frac{{w_pa(n)}\left (\sigma_{2s-1}(n)-\sigma_{2s-1}\left (\frac{n}{p}\right ) \right )}{n^{t+2s-1}}$$
or, with Lemma 1 of \cite{Sh2},
$$-\frac{\Gamma(t)}{(2\pi)^{t-1}}\frac{ \sqrt{\pi} \Gamma(s-\frac12)}{p^s(1-p^{-2s})\Gamma(s)\zeta(2s)}\frac{{w_pL_f(t)L_f(t-1+2s)}}{\zeta(2t+2s-2)}{\frac{1-a(p)p^{1-t-2s}}{1-p^{2-2t-2s}}}.$$
{The functional equation of $L_f$ gives $w_pL_f(1)=-L_f(1)$. Since, further, $a(p)=-w_p$,}
evaluating at $t=1$ and using that $\phi_{\infty 0}^*(s)=-\phi_{0 \infty}^*(s)$, we deduce \eqref{phi*0}.

To show \eqref{boundphi*0}, we first note that, by convexity, we have, for $\Re(s)=\frac12-\varepsilon,$ ${L_f(2s)} \ll |\Im(s)|^{\frac12+\varepsilon}$ and $\zeta(2s) \ll |\Im(s)|^{\varepsilon}$ as  $|\Im(s)| \to \infty$. 
Further, from Stirling's formula, $\frac{\Gamma(s-\frac12)}{\Gamma(s)} \ll |t|^{-1/2}$ and $\frac{\Gamma(1-s)}{\Gamma \left (s \right )} \ll |t|^{2\varepsilon}$. All other terms in the RHS of \eqref{phi*0} are bounded when $\Re(s)=\frac12-\varepsilon$ and thus we deduce the result.
\end{proof}

\section{A first application of the Tauberian method}\label{Tauberian-the-first}
In this section, we discuss two first applications exploiting only the pole of $E^*_{\infty}(z, s)$ at $s=1$, thus avoiding, in the first instance, the complications of the infinitely many poles on the line $\Re(s)=1/2$. 

To this end we will need the location and residues of $E_{\ba}^*(z, s)$. Those have been determined in \cite{O, OTh, P}. 
Let $\{\eta_j\}$ be a complete orthonormal basis of Maass-Hecke cusp forms, normalised according to the Petersson scalar product, with corresponding eigenvalues $\lambda_j = s_j(1-s_j)$ ($j=0, \dots$) of the Laplacian.  We normalise so that $\Re(s_j) \ge \frac12$. Each $\eta_j$ has the Fourier expansion $$\eta_j(z)=\sum_{n \in \mathbb Z \setminus \{0\}} \rho_j(n) \sqrt{y}K_{s_j-\frac12}(2 \pi |n|y) e^{2 \pi i n x}.$$
We will write $s_j=\frac12-it_j$. We label $\lambda_j$ so that $0=\lambda_0 <\lambda_1 \le \dots \le \lambda_M \le \lambda_{M+1} \le \dots$, where $\lambda_1, \dots \lambda_M$ are the exceptional eigenvalues satisfying $0<\lambda_j<\frac14$, or equivalently $t_j=ir_j$ with ${r}_j  \in (0, \frac12)$. For $j \ge M+1$, we have $t_j \ge 0.$

With this notation, for $\Re(s) \ge \frac12$, the poles of $E_{\ba}^*(z, s)$ are simple and contained in
$$\left \{ \frac12 \right \} \cup \bigcup_{j \ge 0} \{s_j, \bar s_j\}.$$ The residue at $s=s_0=1$ is $F_{\ba}(z)/\text{Vol}(\Gamma \backslash \mathbb H)$, where
\begin{equation}\label{anti}F_{\ba}(z)=2 \pi i \int_{\ba}^z f(w)dw.\end{equation}
{We recall the definition of the Rankin-Selberg function, 
$$
L_{f\otimes\eta_\ell}(s):=-\sum_{r\ge1}
\frac{a(r)\overline{\rho_\ell(r)}}{r^{s+1/2}},
$$
initially where the series converges absolutely, and elsewhere by meromorphic continuation. (We have chosen the above normalisation to simplify the formulas.)
Then, for $j\ge1$, the residue of
$E_\infty^*(z,s)$} at $s=s_j$ (resp. $s=\bar s_j$) is
$$R^{+}_j(z)\coloneqq \frac{\Gamma \left(s_j-\frac12 \right ) }{2 \pi^{s_j-\frac12}}\sum_{\ell; \lambda_{\ell}=\lambda_j} L_{f \otimes \eta_{\ell}}(s_{\ell}) \eta_{\ell}(z) \qquad \text{(resp.}  \,  R^{-}_j(z)\coloneqq \frac{\Gamma \left(\bar s_j-\frac12 \right ) }{2 \pi^{\bar s_j-\frac12}}\sum_{\ell; \lambda_{\ell}=\lambda_j} L_{f \otimes \eta_{\ell}}(\bar s_{\ell}) \eta_{\ell}(z) \text{)}.$$

\subsection{The Fourier coefficient $\phi^*(n, s)$ with $n \neq 0$} 
From \eqref{n} and \eqref{C}, for each $n \neq 0$, we have 
$$\Gamma(s) \phi^*(n, s)=\frac{1}{|n|}\sum_{c>0} \frac{S^*(n, 0; pc)}{(p^2c^2/(\pi |n|))^{s}}.$$
For each $r=1, \dots$, we recall the following identity (\cite[7.3 (20)]{erd}):
\begin{equation} \label{std-taub-int}
\frac{1}{2 \pi i}\int_{(1+\varepsilon)}\frac{x^wdw}{w(w+1) \dots (w+r)}=\begin{cases} \frac{1}{r!} \left (1-\frac{1}{x} \right )^r \quad \text{if $x>1$} \\
0 \qquad \text{if $0<x<1$}.\end{cases} 
\end{equation}
Let \be\label{F_N}F
(n; x, s)\coloneqq \frac{\Gamma(s) \phi
^*(n; s)x^s}{s(s+1) \dots (s+r)}.
\ee
Then \eqref{std-taub-int} implies that
\begin{equation}\label{Taube1} \frac{1}{2 \pi i}\int_{(1+\varepsilon)}F
(n; x, s)ds= \frac{1}{r!|n|}\sum_{0<c < \frac{\sqrt{\pi |n|x}}{p}} S^*(n, 0; pc)  \left (1-\frac{p^2c^2}{\pi |n| x} \right )^r.
\end{equation}
For $\varepsilon$ small enough to avoid any possible poles of $E^*(z, s)$ in $(\frac12, 1)$ due to the exceptional eigenvalues, we have
 \begin{equation}\label{shift} \frac{1}{2 \pi i}\int_{(1+\varepsilon)}F
(n; x, s)ds=
 \text{Res}_{s=1} F
(n; x, s)+
\frac{1}{2 \pi i}\int_{(1-\varepsilon)}F
(n; x, s)ds
\end{equation}
The residue of $E^*(z, s)$ at $s=1$ can be written as 
$$\frac{2 \pi i}{\text{Vol}(\Gamma \backslash  \mathbb H)}\int_{\infty}^z f(w)dw=
\frac{1}{ \text{Vol}(\Gamma \backslash  \mathbb H)}\sum_{n \ge 1}\frac{a(n) e^{2 \pi i n z}}{n}. $$ 
From this and \cite[(10.39.2)]{DLMF} we deduce that the residue of $\Gamma(s) \phi^*(n; s)$ at $s=1$ is
\begin{equation}\label{Ress=1}\text{Res}_{s=1}( \phi^*(n; s))=
\frac{1}{2 \sqrt{|n|}K_{\frac12}(2 \pi |n|)}\int_0^1 \text{Res}_{s=1}(E^*(x+i, s))e^{-2 \pi i n x}dx=\frac{a(n)}{n \text{Vol}(\Gamma \backslash  \mathbb H)}\end{equation}
if $n>0$ and $0$, if $n<0.$
On the other hand, because of \eqref{boundphi}, 
\begin{equation}\label{boundleft}\int_{(1-\varepsilon)}F
(n; x, s)ds \ll x^{1-\varepsilon} \int_{-\infty}^{\infty}\frac{(1+|t|)^{\frac12+11\varepsilon}dt}{(1+|t|)^{r+{1}}} \ll  x^{1-\varepsilon} 
\end{equation}
whenever $r \ge 2$. 
Applying \eqref{Taube1}, \eqref{Ress=1} and \eqref{boundleft} to \eqref{shift} we deduce the following
\begin{thm}\label{0thmn} Let $n, r \in \mathbb Z$ with $n \neq 0$ and $r \ge 2.$ Then
\begin{equation*}\label{residfin} \sum_{0<c < \frac{\sqrt{\pi |n| x}}{p}} S^*(n, 0; pc) \left (1-\frac{p^2c^2}{\pi |n| x} \right )^r=\begin{cases}\frac{a(n)x}{(r+1) \text{Vol}(\Gamma \backslash  \mathbb H)}+O(x^{1-\varepsilon}) \quad \text{if $n>0$}\\
O(x^{1-\varepsilon}) \quad \text{if $n<0$.} \end{cases}
\end{equation*} 
\end{thm}

\subsection{The coefficient $\phi^*_{0 \infty}(s)$.} 
Since $\phi^*_{\ba \ba}(s)=0$ for all cusps $\ba$, we focus on $\phi^*_{0 \infty}(s)=-\phi^*_{\infty 0}(s).$ 
 
From \cite[(7.3(20))]{erd}, we deduce
\begin{equation} 
\int_{(1+\varepsilon)}\frac{\G\left (s-\frac12\right )}{\G(s)s \dots (s+r)}x^sds=
\int_{(1+\varepsilon)}\frac{\G\left (s-\frac12\right )}{\G(s+r+1)}x^sds=
\begin{cases} \frac{2 \pi i \sqrt{x}}{\G(r+\frac32)}\left ( 1-\frac{1}{x} \right ) ^{r+\frac12} \quad \text{if $x>1$} \\
0 \qquad \text{if $0<x<1$.}\end{cases}
\end{equation}

This, together with \eqref{0}, implies
\begin{equation}\label{Taube0} \int_{(1+\varepsilon)}\frac{\phi_{0 \infty}^*(s)x^sds}{s \dots (s+r)}= \frac{2 i\sqrt{\pi^3 x}}{\sqrt{p} \Gamma(r+\frac32)}\sum_{\substack{0<c < \sqrt{\frac{x}{p}} \\ { (c, p)=1}}} \frac{S^*_{0 \infty}(0, 0; \sqrt{p} c)}{c} \left ( 1-\frac{p c^2}{x}\right )^{r+\frac12} .
\end{equation}
As in the previous subsection, we will move the line of integration to $\Re(s)=1-\varepsilon$.
The pole of $E^*_{0}(z, s)$ at $s=1$ has residue $$\frac{2 \pi i}{\text{Vol}(\Gamma \backslash  \mathbb H)}\int_{0}^z f(w)dw=\frac{1}{\text{Vol}(\Gamma \backslash  \mathbb H)}\left ( { -2 \pi} \Lambda_f(1)+\sum_{n \ge 1}\frac{a(n) e^{2 \pi i n z}}{n} \right ) . $$ From this  we deduce that the residue of $\phi_{0 \infty}^*(s)$ at $s=1$ is
\begin{equation}\label{Ress=10}\text{Res}_{s=1}(\phi_{0\infty}^*(s))=\int_0^1 \text{Res}_{s=1}(E_{0}^*(x+i, s))dx=\frac{{-2 \pi}\Lambda_f(1)}{\text{Vol}(\Gamma \backslash  \mathbb H)}.\end{equation}
Therefore, with \eqref{boundphi}, \eqref{Taube0} equals
\begin{equation*}
\frac{{-4 \pi^2 i} \Lambda_f(1)x}{(r+1)! \text{Vol}(\Gamma \backslash  \mathbb H)}+\int_{(1-\varepsilon)}\frac{ \phi_{0 \infty}^*(s)x^sds}{s \dots (s+r)}= \frac{{-4 \pi^2 i} \Lambda_f(1)x}{(r+1)! \text{Vol}(\Gamma \backslash  \mathbb H)}+O\left ( \int_{(1-\varepsilon)}\frac{ (1+|t|)^{11\varepsilon} x^{1-\varepsilon} ds}{(1+|t|)^r}\right )
\end{equation*}
 for $r>1$ which implies
\begin{thm}\label{0thm0} Let $r \in \mathbb N$ with $r \ge 2.$ Then
\begin{equation*}\label{residfin0} \sum_{\substack{0<c < \sqrt{\frac{x}{p}} \\ { (c, p)=1}}} \frac{S^*_{0 \infty}(0, 0; \sqrt{p}c)}{c} \left ( 1-\frac{pc^2}{x}\right )^{r+\frac12}=
\frac{{-2 (\pi p x)^{\frac12}} \Lambda_f(1)\Gamma \left ( r+\frac32\right ) }{(r+1)! \textnormal{Vol}(\Gamma \backslash  \mathbb H)}+
O \left (x^{\frac12-\varepsilon} \right ).
\end{equation*} 
\end{thm}

\section{A  Tauberian-type theorem for the modified Kloosterman sums $S^*$.}\label{Tauberian-the-second}
In addition to the pole at $s=1$ and the possible poles due to the exceptional eigenvalues, $E^*(z, s)$ has infinitely many poles on the line $\Re(s)=1/2$ (\cite{OTh}, see also \cite{G}). To get a stronger estimate than those obtained in Theorems \ref{0thmn} and \ref{0thm0}, we need to consider those infinitely many poles. 
We first define the following, for any $n \in \mathbb Z \setminus \{0\}$ and $j>0$:
\begin{align*}R^{+}_j(n)\coloneqq 
\frac{\G(s_j-\frac12)  \G(s_j)^2 \zeta(2s_j)^2 I_{s_j-\frac12}(2 \pi |n|) K_{s_j-\frac12}(2 \pi |n|)}{4 \sqrt{|n|}\pi^{s_j-\frac12} \G(s_j+k)}\sum_{\substack{\ell : \lambda_\ell=\lambda_j}} L_{f \otimes \eta_{\ell}}(s_j)\rho_\ell(n) 
\end{align*}
and $R^{-}_j(n)\coloneqq \overline{R^{+}_j(n).}$ In the sums above, $j$ indexes distinct spectral parameters $s_j$ and $\ell$ indexes an orthonormal basis of the corresponding eigenspace.
All $R_j^{\pm}$-sums below omit $s_j=\frac12$. For $F_{n,k}$ in \eqref{Fnk}, set $H_{n,k}(s)\coloneqq(s-\tfrac12)^3F_{n,k}(1,s)$, holomorphic at $\frac12$ as shown below, and
\begin{equation}\label{E^star}
\mathcal E^*(n,k,X)\coloneqq \frac{\sqrt X}{2}
\left(H_{n,k}''(\tfrac12)+2H_{n,k}'(\tfrac12)\log X
+H_{n,k}(\tfrac12)(\log X)^2\right).
\end{equation}

\begin{thm}\label{main} Let $f(z)=\sum a(n) q^n$ be a {normalised Hecke eigenform} of weight $2$ for $\Gamma=\Gamma_0(p)$ ($p$ prime) and let $S^*(m, n; c)$ denote the Kloosterman sum attached to it in \eqref{Kl*}. Further, let $k \in \mathbb N$ with $k > 6$ and let $X>1$ be such that $\pi X \not \in \mathbb Q.$ Then for every $n \in \mathbb Z \setminus \{0\},$
\begin{multline}\label{final}
\sum_{c<\frac{\sqrt{\pi X |n|}}{p}} 
S^*(n, 0; pc) \psi_c(n, k, X) =
\frac{\pi^{3} a(n) (1-e^{-4 \pi |n|})}{144 n^2 k! \text{Vol}(\Gamma \backslash  \mathbb H)}X+\sum_{j=1}^M R^{+}_j(n)X^{s_j}
\\ + \sum_{j>M} \left ( R^{+}_j(n)X^{s_j}+R^{-}_j(n)X^{\bar s_j} \right )+
\mathcal E^*(n,k,X)
+ O(X^{\frac12-\varepsilon}), 
\end{multline}
 $$\text{where} \, \, \psi_c(n, k, X)\coloneqq  \frac12\sum_{m <\frac{\sqrt{\pi X |n|}}{pc}} \frac{\sigma_0(m)}{|n|(k-1)!}
\int_{\frac{(pmc)^2}{\pi X |n|}}^1 \left (1-\frac{(pmc)^2}{\pi X |n|t} \right )^{k-1} J_0(2 \pi |n| (\sqrt{t}^{-1}-\sqrt{t})) \frac{dt}{t^{\frac32}}.$$
Here $a(n)$ is taken to be $0$ if $n<0$ and $j$ indexes $s_j$ without multiplicity. Further the sum over $j$ in \eqref{final} is absolutely convergent.
\end{thm}
{\it Remarks.} 1. { Since $|J_0(u)| \le 1$ and $\sum_{m <A} \sigma_0(m)/m \le \sum_{ab < A}(ab)^{-1} \le (1+\log A)^2,$} the weight $\psi_c(|n|, k, X)$ is {$\ll_{n, k} X^{\frac12} (1+\log(\sqrt{\pi X |n|}/(pc)))^2/(pc)$}. \\
2.  Notice that, when $a(n)=0$ (and, in particular, when $n<0$), the first term in the RHS of \eqref{final} disappears. This means that the weighted sums accumulate when $n>0$ and $a(n)\neq0$ but not when $n<0$. If, further, there are no exceptional eigenvalues 
{then} the order goes down to 
 {$O(X^{1/2}(1+(\log X)^2))$}.  

We provide numerical information related to Theorem~\ref{main} and make related conjectures in Sections~\ref{num-Taub}-\ref{Maasscontr} below.

\subsection{Estimates for $E^*(z, s)$ and its Fourier coefficients}
To bound $R_j(z)$, defined above, we recall two bounds. First, by \cite{PS}, \cite{DI}, we have, for $p$ prime,
\begin{equation}\label{RSbound} L_{f \otimes \eta_j}(s_j) \ll_{p, \varepsilon} (1+|t_j|)^{\frac12+\varepsilon} e^{\frac{\pi |t_j|}{2}}.
\end{equation}
Further, by \cite[Prop. 1]{BH}, 
\begin{equation}\label{supnorm} ||\eta_j||_{\infty} \ll_{\varepsilon} p^{\varepsilon}  (1+|t_j|)^{\frac{5}{12}+\varepsilon}.
\end{equation}
Together with Weyl's law and Stirling we deduce that, for the infinitely many $s_j$ with $\Re(s_j)=\frac12$, 
\begin{equation}\label{Rjbound} 
R^{\pm}_j(z) \ll e^{-\frac{\pi |t_j|}{2}} (1+|t_j|)^{-\frac12} (1+|t_j|)^{1+\frac12+\varepsilon}  e^{\frac{\pi |t_j|}{2}}  (1+|t_j|)^{\frac{5}{12}+\varepsilon} =(1+|t_j|)^{\frac{17}{12}+2\varepsilon}.
\end{equation}
We set, for $z \in \mathbb H$ and $s \in \mathbb C$,
\begin{equation}\label{DefEo} 
E^o(z, s)\coloneqq \sum_{j>M} \left ( \frac{R^+_j(z) e^{(s-s_j)^2}}{s-s_j} +\frac{R^-_j(z) e^{(s-\bar s_j)^2}}{s-\bar s_j}\right ).
\end{equation}
\begin{lem} \label{Rbound} For  $z \in \mathbb H$ and $s$ at least $\varepsilon$ away from all $s_j$, $\bar s_j$ $E^o(z, s)$ converges absolutely and it satisfies 
$$E^o(z, \sigma+it)=O((1+|t|)^{\frac{41}{12}+\varepsilon}) \quad \text{for $t$ at least $\varepsilon$ away from any $\pm t_j$}$$ uniformly in $z$ and for $\sigma$  in $[\frac12-\varepsilon, 1+\varepsilon].$
\end{lem}
\begin{proof}
By \eqref{Rjbound} and the Gaussian decay of $e^{(s-s_j)^2}$, $E^o$ converges absolutely for each $s \neq s_j, \bar s_j$.
Again with \eqref{Rjbound}  
\begin{equation}\label{EoBig} 
\sum_{|t_j| > 2 |t|} \left | \frac{R^+_j(z) e^{(s-s_j)^2}}{s-s_j} \right | \ll \sum_{|t_j| > 2 |t|}  \frac{|t_j|^{\frac{17}{12}+\varepsilon} e^{-\frac{|t_j|^2}{4}}}{|t_j|}=O(1), \quad \text{since $|t+t_j| \ge |t_j|/2$.}
\end{equation}
Finally, with Weyl's law, we obtain
\begin{equation}\label{EoSmall} 
\sum_{|t_j| \le 2 |t|} \left | \frac{R^+_j(z) e^{(s-s_j)^2}}{s-s_j} \right | \ll \sum_{|t_j|  \le 2 |t|}  (1+|t_j|)^{\frac{17}{12}+\varepsilon} \ll (1+|t|)^{2+\frac{17}{12}+\varepsilon}.
\end{equation}
The terms corresponding to $R^-_j(z)$ are treated analogously.
\end{proof}
Set
\begin{equation}\label{DefEtilde} 
\tilde E(z, s)\coloneqq E^*(z, s)- E^o(z, s) .
\end{equation}
The function $\zeta(2s)^2\tilde E(z,s)$ is holomorphic in the upper and lower half-strips $\frac12-\varepsilon\le\Re(s)\le1+\varepsilon$, $|\Im(s)|>\varepsilon$, with $\varepsilon$ small enough to exclude reflected exceptional poles from the strip. Indeed, $E^o$ cancels the cuspidal principal parts on $\Re(s)=\frac12$ away from the real axis and the factor $\zeta(2s)^2$ is added to cancel any poles due to the scattering matrix. Therefore,Phragm\'en-Lindel\"of
holds and we have the following.
\begin{prop} \label{phibound} For $|t|>\varepsilon$ and at least $\varepsilon$ away from any $\pm t_j$, we have
\begin{equation}\label{boundphi<1/2}\zeta(2s)^2 K_{s-\frac12}(2 \pi |n|) \phi^*(n, s)=O_{f, n, \varepsilon}( (1+|t|)^{5
+\varepsilon}). \end{equation}
uniformly for $\Re(s)$ in $[\frac12-\varepsilon, 1+\varepsilon].$
\end{prop}
\begin{proof} We will first apply convexity to $\zeta(2s)^2 \tilde E(z, s)$ in the strip $\Re(s) \in [\frac12-\varepsilon, 1+\varepsilon]$ and we will then combine it with the last lemma to bound $\zeta(2s)^2 E^*(z, s).$ From this, we will deduce the claimed bound for the Fourier coefficients of $E^*(z, s)$.

Since $E^*(z, s)$ and $\zeta(2s)$ are absolutely convergent for $\Re(s)>1$ \cite[Prop. 2.6]{PR}, Lemma \ref{Rbound} implies that, for $\Re(s)=1+\varepsilon$, $\zeta(2s)^2 \tilde E(z, s) \ll (1+|t|)^{\frac{41}{12}+\varepsilon}$, uniformly for $z$ in a compact set $K.$ 

To bound $\tilde E(z, s)$ on $\Re(s)=\frac12-\varepsilon$, we use the functional equation \eqref{FE} to obtain
\begin{multline}\label{FEspec} \zeta(2s)^2 E^*(z, s)=\frac{\pi^{2s-1}(p-1)}{p^{2s}-1}\frac{\Gamma(1-s)}{\Gamma(s)}\zeta(2-2s) \zeta(2s) E^*(z, 1-s)\\
+ \frac{\pi^{2s-1}(p^s-p^{1-s})}{(p^{2s}-1)} \frac{\Gamma(1-s)}{\Gamma(s)}\zeta(2-2s) \zeta(2s)  E^*_0(z, 1-s)+\phi^*_{\infty \, 0}(s) \zeta(2s)^2 E_0(z, 1-s).
\end{multline}
For $\Re(s)=\frac12-\varepsilon$, we have $\zeta(2s) \ll |t|^{\varepsilon}$ and $\zeta(2-2s) \ll 1$, as $|t| \to \infty$. Therefore, with ${|}\Gamma(1-s)/\Gamma(s){|} \sim |t|^{2 \varepsilon}$ and \eqref{boundE*E}, we deduce that, as $|t| \to \infty$, the first and second terms in the RHS of \eqref{FEspec} are $|t|^{5+\varepsilon},$ for some $\varepsilon>0$. For the last term of \eqref{FEspec}, we use \eqref{boundphi*0} and \eqref{boundE*E}, to deduce that it is  $|t|^{2+\varepsilon}.$ Therefore,  for $\Re(s)=\frac12-\varepsilon$,  we have,
\begin{equation}\label{boundE*1/2} \zeta(2s)^2 E^*(z, s) \ll |t|^{5+\varepsilon}+ |t|^{2+\varepsilon} \ll  |t|^{5+\varepsilon}, \qquad \text{as $|t| \to \infty$}\end{equation}
uniformly for $z$ in a compact set $K.$ Because of Lemma \ref{Rbound} and $\zeta(1-\varepsilon+it)=O(|t|^{\varepsilon}),$ we deduce the same bound for $\zeta(2s)^2 \tilde E(z, s).$

We now apply Phragm\'en-Lindel\"of to $\zeta(2s)^2 \tilde E(z, s)$ in the upper (resp. lower) strip defined by $\Re(s) \in [\frac12-\varepsilon, 1+\varepsilon]$ and $|t|>\varepsilon$. Since  $\zeta(2s)^2 \tilde E(z, s)$ is holomorphic there, we deduce that 
\begin{equation}\label{boundEtilde} \zeta(2s)^2 \tilde E(z, s) \ll |t|^{5
+\varepsilon} \qquad \text{as $|t| \to \infty$}\end{equation}
uniformly for $\Re(s)$ in $[\frac12-\varepsilon, 1+\varepsilon]$ and for $z$ in a fixed compact set $K.$  Therefore, with Lemma \ref{Rbound}, we deduce that, for $|t|>\varepsilon$ and at least $\varepsilon$ away from any $\pm t_j$, we have
\begin{equation}\label{boundE*} \zeta(2s)^2 E^*(z, s) \ll |t|^{
5+\varepsilon}+ |t|^{\frac{41}{12}+\varepsilon} \ll |t|^{
5+\varepsilon}, \end{equation} uniformly for $\Re(s) \in [\frac12-\varepsilon, 1+\varepsilon]$ and for $z$ in a compact set $K.$ The proposition follows from the identity $K_{s-\frac12}(2 \pi |n|) \phi^*(n, s)=\frac{1}{2\sqrt{|n|}}\int_0^1 E^*(x+i, s)e^{-2 \pi i n x}dx$.
\end{proof}

\subsection{Proof of Theorem \ref{main}.}
For $X>1,$ we will study the integral 
$$I=\frac{1}{2 \pi i} \int_{(1+\varepsilon)} F_{n, k}(X, s)ds, \qquad   \text{where}$$
 \begin{equation}\label{Fnk}
F_{n, k}(X, s)\coloneqq \frac{\zeta(2s)^2 \Gamma(s)^2}{\Gamma(s+k)} \phi^*(n, s) K_{s-\frac12}(2 \pi |n|) I_{s-\frac12}(2 \pi |n|) X^s.
\end{equation}
We first compute it directly to obtain Riesz-type sums involving the modified Kloosterman sums $S^*(m, 0, pc)$.

By \eqref{bound2plus}, Lemma \ref{boundK} and the exponential decay of $\Gamma(s)$, we deduce that $I$ is absolutely convergent and, in particular, we can interchange summation and integration to derive, from \eqref{n} and the series expression for $\zeta(2s)^2$,
\begin{equation}\label{I1Th}I=\sum_{c, m>0} \frac{S^*(n, 0; pc) \sigma_0(m)}{2 \pi i |n|} \int_{(1+\varepsilon)} F_{n, k}(c, m; X, s) ds, \end{equation}
where 
$$F_{n, k}(c, m; X, s)\coloneqq \frac{\Gamma(s)}{\Gamma(s+k)} K_{s-\frac12}(2 \pi |n|) I_{s-\frac12}(2 \pi |n|) \left ( \frac{\pi X |n|}{(pcm)^2}\right )^s.$$

With \cite[(10.32.16)]{DLMF} and the change of variables $y=e^{-2t}$ we have
\begin{multline*} K_{s-\frac12}(2 \pi |n|) I_{s-\frac12}(2 \pi |n|)=\int_0^\infty  J_0(4 \pi |n| \sinh  t)e^{-2 \left ( s-\frac12\right) t} dt \\=
\int_0^1 \frac{J_0(2 \pi |n| (y^{-\frac12}-y^{\frac12}))}{2 \sqrt{y}} y^{s-1}dy=\mathcal M \left [g \right ](s),
\end{multline*}
where
$$g(y)\coloneqq \mathbf 1_{(0, 1)}(y) \frac{J_0(2 \pi |n| (y^{-\frac12}-y^{\frac12}))}{2\sqrt{y}}$$
and $\mathbf 1_{(0, 1)}(x)$ is the characteristic function of $(0, 1)$.  Further, with \eqref{std-taub-int}, we have $\Gamma(s)/\Gamma(s+k)=\mathcal M \left [h \right ](s),$ where
$$h(y)\coloneqq \mathbf 1_{(0, 1)}(y) \frac{(1-y)^{k-1}}{(k-1)!} $$
Thus, for $Y=(pcm)^2/(\pi X |n|)$,
$$F_{n, k}(m; X, s) =\mathcal M \left [h \right ](s) \mathcal M \left [g \right ](s)Y^{-s}.$$ 

If $h \star g$ denotes the multiplicative convolution, then the Mellin convolution theorem implies that
\begin{multline}\frac{1}{2 \pi i } \int_{(1+\varepsilon)}F_{n, k}(c, m; X, s) ds =
\mathcal M^{-1} \left ( \mathcal M \left [h \right ](s) \mathcal M \left [ g \right ](s)\right )(Y)=(h \star g)(Y)\\  
=\int_0^{\infty}\mathbf 1_{(0, 1)}\left ( \frac{Y}{t}\right )\frac{\left (1-\frac{Y}{t} \right )^{k-1}}{(k-1)!} \, \mathbf 1_{(0, 1)}(t) 
\frac{J_0(2 \pi |n| (t^{-\frac12}-t^{\frac12}))}{2\sqrt{t}} \frac{dt}{t} \\=
\begin{cases} 0 \qquad \qquad \text{if $Y \ge 1$} \\ \int_Y^{1}\frac{\left (1-\frac{Y}{t} \right )^{k-1}}{(k-1)!} \frac{J_0(2 \pi |n| (\sqrt{t}^{-1}-\sqrt{t}))}
{2\sqrt{t}} \frac{dt}{t} \, \, \text{if $0<Y<1.$} \end{cases}
\end{multline}

With \eqref{I1Th}, we deduce that 
\begin{equation}\label{V+} I={\tfrac12}\sum_{c<\frac{\sqrt{\pi X |n|}}{p}} 
\frac{S^*(n, 0; pc)}{|n|} \sum_{m=1}^{\lfloor \frac{\sqrt{\pi X |n|}}{pc} \rfloor} \frac{\sigma_0(m)}{(k-1)!}
\int_{\frac{(pmc)^2}{\pi X |n|}}^1 \left (1-\frac{(pmc)^2}{\pi X |n|t} \right )^{k-1} J_0(2 \pi |n| (\sqrt{t}^{-1}-\sqrt{t})) \frac{dt}{t^{\frac32}}.
\end{equation}

To estimate $I$, we will shift the line of integration to $\Re(s)=\frac12-\varepsilon$. Because of the infinitely many poles on the line $\Re(s)=\frac12$, this must be done carefully. 

{Choose $Y_\ell\in[\ell,\ell+1]$ with $\operatorname{dist}(Y_\ell,\{\pm t_j:j>M\})\gg\ell^{-2}$. This is possible by $N(T)\ll_p T^2$ because for a sufficiently small fixed $c>0$, the intervals of radius $c\ell^{-2}$ around the spectral heights in $[\ell-1,\ell+2]$ have total length less than one.} We consider the horizontal (resp. vertical) oriented segments $H_\ell^{+}\coloneqq [\frac12-\varepsilon+iY_\ell, 1+\varepsilon +i Y_\ell], H_\ell^{-}\coloneqq [ \frac12-\varepsilon-iY_\ell, 1+\varepsilon -i Y_\ell]$, $V_\ell^{+}\coloneqq [1+\varepsilon-iY_\ell, 1+ \varepsilon +iY_\ell]$ and $V_\ell^{-}\coloneqq [\frac12-\varepsilon-iY_\ell, \frac12 - \varepsilon +iY_\ell]$.
We have 
\be\label{Y_N} \frac{1}{2 \pi i}\int_{V_\ell^+}  F_{n, k}
(X, s)ds=
\frac{1}{2 \pi i} \int_{V_\ell^-} F_{n, k}
(X, s)ds+ R_\ell+
\frac{1}{2 \pi i} \left (\int_{H^+_\ell}-\int_{H^-_\ell} \right ) F_{n, k}
(X, s)ds
\ee
where
$R_\ell$ is the sum of the residues of $F_{n, k}
(X, s)$ at all poles $s$ with $|\Im(s)|<Y_\ell$ and $\Re(s) \in (\frac12-\varepsilon, 1+\varepsilon)$.

{On $H_\ell^{\pm}$, \eqref{Rjbound}, the Weyl bound and Gaussian decay give
\[
E^o(z,s)\ll Y_\ell^{17/12+2\varepsilon}\,Y_\ell^2\,Y_\ell^2+1
\ll Y_\ell^{65/12+2\varepsilon}.
\]
(The two factors $Y_\ell^2$ reflect the number of nearby heights and the distance to the poles.) By convexity, $\zeta(2s)^2\ll Y_\ell^{3\varepsilon}$, so \eqref{boundEtilde} and \eqref{eq:phi_ns} imply
$\zeta(2s)^2K_{s-\frac12}(2\pi|n|)\phi^*(n,s)\ll Y_\ell^{65/12+5\varepsilon}$.
Stirling's formula and Lemma \ref{boundK} therefore give
\begin{equation}\label{H+}
\left|\int_{H_\ell^{\pm}}F_{n,k}(X,s)\,ds\right|
\ll X^{1+\varepsilon}Y_\ell^{59/12-k+5\varepsilon},
\end{equation}
which tends to zero for fixed $X>1$, $k>6$ and sufficiently small $\varepsilon>0$.}

{On $\Re(s)=\frac12-\varepsilon$, use \eqref{boundE*1/2} and \eqref{eq:phi_ns}. Stirling's formula and Lemma \ref{boundK} give
\[
F_{n,k}(X,\tfrac12-\varepsilon+it)
\ll X^{1/2-\varepsilon}(1+|t|)^{9/2-k+\varepsilon}.
\]
This is integrable for $k\ge6$ and sufficiently small $\varepsilon>0$, giving}
\be \label{V_} \lim_{\ell \to \infty}\int_{V_\ell^-}  F_{n, k}
(X, s)ds=
\int_{(\frac12-\varepsilon)}  F_{n, k}
(X, s)ds=O(X^{\frac12-\varepsilon}).
\ee
Applying \eqref{V+}, \eqref{H+} and \eqref{V_} to \eqref{Y_N}, we deduce
\begin{equation}\label{finalint}
\sum_{c<\frac{\sqrt{\pi X |n|}}{p}} 
S^*(n, 0; pc) \psi_c(n, k, X) =\lim_{\ell} R_{\ell}+O(X^{\frac12-\varepsilon}).
\end{equation}
To compute $\lim_{\ell} R_{\ell}$ we recall that the crossed poles are $1$, $s_j$ ($1 \le j \le M$), $s_j, \bar s_j$ ($j>M$) and $1/2.$
For the computation of the residues at each but the last pole, we use \eqref{Fnk} and the residue of $\phi^*(n, s)$ inherited by the residue of $E^*(z, s)$. 

For the pole at $1/2$ we note that by the Fourier expansion we have, for $n\neq 0$,
\begin{equation}\label{eq:phi_ns}
K_{s-\frac12}(2\pi |n|)\phi^*(n,s)=\frac{1}{2\sqrt{|n|}}\int_0^1 E^*(x+i,s)e^{-2\pi i n x}\,dx .
\end{equation}
{Put $Q=E^*-F_{\infty}E_{\infty}$, with $F_{\infty}$ as in \eqref{anti}. By \eqref{inde} and the cusp expansions, $Q$ is automorphic and $L^2$ for $\Re(s)>\frac12$, away from poles. Direct differentiation gives
\[
Q=(\Delta-s(1-s))^{-1}G,\qquad \text{with
$G=8\pi i\,y^2f\,\partial_{\bar z}E_{\infty}.$}
\]
Since \eqref{scat} gives $E_{\ba}(z,\frac12)=0$, the rapidly decaying source $G=O(s-\frac12)$, and \cite[Theorem 2]{P94} gives, by meromorphic continuation,
\[
Q=\left(\frac{P_{1/4}}{(s-\frac12)^2}+O(1)\right)G,
\]
where $P_{1/4}$ projects orthogonally onto the cuspidal $\frac14$-eigenspace (or is zero). Thus $E^*$ has at most a simple pole at $\frac12$; by \eqref{Fnk} and \eqref{eq:phi_ns}, $F_{n,k}(1,s)$ has at most a triple pole. Hence $H_{n,k}$ is holomorphic and, upon expanding $X^s=\sqrt X\exp((s-\frac12)\log X)$, we deduce that $\operatorname*{Res}_{s=\frac12}F_{n,k}(X,s)=\mathcal E^*(n,k,X)$.
If $E^*$ is holomorphic and vanishes at $\frac12$ (in particular, if $P_{1/4}=0$), this reduces to
\[
\mathcal E^*(n,k,X)=
\frac{\pi I_0(2\pi|n|)X^{1/2}}{8\sqrt{|n|}\Gamma(k+\frac12)}
\int_0^1\left.\partial_s E^*(x+i,s)\right|_{s=\frac12}
e^{-2\pi i n x}\,dx.
\]

To verify the absolute convergence of the sum over $j$ in the right-hand side of \eqref{final}, we first use the bound $\rho_{\ell}(n) \ll_n e^{\frac{\pi |t_j|}{2}} |t_j|^{\frac12+\varepsilon}$, 
\eqref{RSbound} and Weyl's Law to deduce
$$\sum_{\substack{\ell : \lambda_\ell=\lambda_j}} L_{f \otimes \eta_{\ell}}(s_j)\rho_\ell(n) \ll |t_j|^{1+\varepsilon} \left (e^{\frac{\pi |t_j|}{2}}  |t_j|^{\frac12+\varepsilon} \right ) \left (  e^{\frac{\pi |t_j|}{2}} |t_j|^{\frac12+\varepsilon} \right )=e^{\pi |t_j|}  |t_j|^{2+3\varepsilon} .$$
Stirling's formula and Lemma \ref{boundK} imply that the remaining factors of the $j$-th term of the sum are $\ll e^{-\pi |t_j|}|t_j|^{-\frac32-k+\varepsilon}$. Therefore, the $j$-th term is $\ll |t_j|^{\frac12-k+\varepsilon}$ and hence, with Weyl's law, the series converges for $k>5/2$. The sum corresponding to $R_j^-(n)$ is treated similarly.

\section{A modified Kloosterman zeta function}
\label{6}
In this section we will show an analogue of the main result of \cite{GS} for 
$$Z^*_{m, n}(s)\coloneqq \sum_{N|c}\frac{S^*(m, n; c)}{c^{2s}}$$ where $S^*(m, n; c)$ denotes the  modified Kloosterman sum defined in \eqref{Kl*}.  From \eqref{boundS*} we deduce that $Z^*_{m, n}(s)$ converges for $\Re(s)>1$.  We will follow the perturbation theory approach of \cite{P}. 

\subsection{Expression of $Z_{m, n}^*$ in terms of Kloosterman zeta functions associated with characters.}
We let $w^1$ (resp. $w^2$) be the real (resp. imaginary) part of $f(z)dz$. For $\epsilon>0$ and $i=1, 2$ we define the \emph{unitary} character of $\Gamma=\Gamma_0(N)$
$$\chi_{\epsilon}^i(\gamma)=\text{exp}\left ( -2 \pi i \epsilon \int_{i \infty}^{\gamma i \infty} w^i\right ).$$
Set
$$Z_{m, n}(s; \chi^i_{\epsilon})\coloneqq \sum_{N|c}\frac{S_{\chi_{\epsilon}^i}(m, n; c)}{c^{2s}}, \quad \text{where, $S_{\chi_{\epsilon}^i}(m, n; c)\coloneqq \sum_{\substack{\g \in \Gamma_{\infty}\backslash\Gamma/\Gamma_{\infty} \\ \, \, c_\g=c}} \chi^i_{\epsilon}(\gamma)e^{2 \pi i \left (n\frac{a_{\g}}{c}+m\frac{d_{\g}}{c} \right )}.$}$$
Since $\chi_{\epsilon}^i$ is unitary, the series $Z_{m, n}(s; \chi^i_{\epsilon})$ converges absolutely and uniformly in $\epsilon$ for $\Re(s) \gg 1$. Likewise, with \eqref{boundmods}, the differentiated series converges absolutely and uniformly in $\epsilon$ for $\Re(s) \gg 1$ and hence we can differentiate $Z_{m, n}(s; \chi^i_{\epsilon})$ term-by-term to get 
$$\partial_{\epsilon}Z_{m, n}(s; \chi^i_{\epsilon}) |_{\epsilon=0}=-2 \pi i \sum_{N|c}\frac{1}{c^{2s}}  \sum_{\substack{\g \in \Gamma_{\infty}\backslash\Gamma/\Gamma_{\infty} \\ \, \, c_\g=c}} \left (\int_{i \infty}^{\gamma i \infty} w^i \right )  e^{2 \pi i \left (n\frac{a_{\g}}{c}+m\frac{d_{\g}}{c} \right )}.$$
Therefore
\begin{equation}\label{Z*-Zchi} Z^*_{m, n}(s)=\partial_{\epsilon}\left (Z_{m, n}(s; \chi^1_{\epsilon})+iZ_{m, n}(s; \chi^2_{\epsilon}) \right ) |_{\epsilon=0}
\end{equation}
In view of \eqref{Z*-Zchi}, we will show the meromorphic continuation of $Z^*_{m, n}$ by establishing the  meromorphic continuation and real-analytic dependence on $\epsilon$ of the functions $Z_{m, n}(s; \chi^i_{\epsilon})$. 

\subsection{Basics of Sobolev spaces}
For the analysis of $Z_{m, n}(s, \chi_{\epsilon}^i)$ as $\epsilon$ varies, and, especially, for its bound, we will use some techniques from the theory of Sobolev spaces. Since these analytic techniques may be less familiar than those used in the previous sections, we outline the background theory and its application to our setting in more detail than above, to keep the exposition self-contained. 

We state a general definition of Sobolev spaces on $X=\Gamma \backslash \mathbb H$ ($\Gamma=\Gamma_0(N)$) in the formulation of \cite[Section 2.1]{Hebey} and we then describe it in more detail in the special case we need it.
Let $U\subset X$ be open.  For a smooth function $u\in C^\infty(U)$, set $
        \nabla^0u\coloneqq u,$ 
and for $j\geq1$, let $\nabla^j u$ denote the $j$-th covariant
derivative with respect to the Levi-Civita connection of the hyperbolic metric.  We write
$$
        \|\nabla^j u\|_{L^2(U)}^2
        \coloneqq 
        \int_U |\nabla^j u|_{\mathrm{hyp}}^2\,d\mu,
$$
where $d\mu$ is the hyperbolic volume form and
$|\cdot|_{\mathrm{hyp}}$ is the pointwise tensor norm induced by the
hyperbolic metric. The detailed definitions of those terms can be found in  \cite[Section 2.1]{Hebey}, but, in the sequel, we will only describe them in the specific cases we will need them. 
 
For an integer $k\geq0$, define
$$
        \mathcal C_k(U)
        \coloneqq 
        \left\{
        u\in C^\infty(U):
        \|\nabla^j u\|_{L^2(U)}<\infty
        \text{ for }0\leq j\leq k
        \right\}.
$$
The Sobolev space $H^k(U)$ is the completion of $\mathcal C_k(U)$ with
respect to the norm
$$
        \|u\|_{H^k(U)}
        \coloneqq 
        \left(
        \sum_{j=0}^k
        \|\nabla^j u\|_{L^2(U)}^2
        \right)^{1/2}.
$$
Equivalently, $H^k(U)$ consists of $L^2$-functions whose covariant
derivatives up to order $k$ are square-integrable in the distributional sense.

For $F \in C^{\infty}(U)$, the first covariant derivative is simply the ordinary
differential.  Thus, in the standard coordinate $z=x+iy$ on a lift to
$\mathbb H$,
$$
        \nabla F=dF=F_x\,dx+F_y\,dy.
$$
Since the hyperbolic metric is $ds^2=(dx^2+dy^2)/y^2$, 
the inverse of its matrix is given by $
        g^{ij}=y^2\delta^{ij}.$
Therefore
$$
        |\nabla F|_{\mathrm{hyp}}^2
        =
        |dF|_{\mathrm{hyp}}^2
        =
        y^2\left(|F_x|^2+|F_y|^2\right).
$$
and, hence, in a coordinate patch
    $$    \|\nabla F\|_{L^2(U)}^2
        =
        \int_U
        y^2\left(|F_x|^2+|F_y|^2\right)
        \frac{dx\,dy}{y^2}
        =
        \int_U
        \left(|F_x|^2+|F_y|^2\right)\,dx\,dy.
$$
Thus, if $U$ is contained in a coordinate patch which has been lifted to
$\mathbb H$, we have
$$    \|F\|_{H^1(U)}^2
        =
        \int_U \frac{|F|^2}{y^2}\,dx\,dy
        +
        \int_U
        \left(|F_x|^2+|F_y|^2\right)\,dx\,dy.$$
If $K$ is a compact subset of $U$, then $y$ is bounded
above and below on $K$ and hence
$$\|F\|_{H^1(K)}^2    \asymp_K
        \int_K        \left(|F|^2+|F_x|^2+|F_y|^2
        \right)\,dx\,dy.
$$
Here $\|F\|_{H^1(K)}$ denotes the local Sobolev norm of the restriction of $F$ to $K$, and the constants depend only on $K$.
Similarly, for open $U\subset X$,
$$
        \|F\|_{H^2(U)}^2
        =
        \|F\|_{L^2(U)}^2
        +
        \|\nabla F\|_{L^2(U)}^2
        +
        \|\nabla^2F\|_{L^2(U)}^2,$$
where $\nabla^2F$ is the covariant Hessian.  As above, on every open $U$ with compact closure contained in a coordinate patch, this norm is equivalent to the following norm:
$$    \|F\|_{H^2(U)}^2
        \asymp_U
        \int_U
        \left(
        |F|^2+|F_x|^2+|F_y|^2
        +|F_{xx}|^2+|F_{xy}|^2+|F_{yy}|^2
        \right)\,dx\,dy.
$$

We will now state two general lemmas (see, e.g. \cite[pg. 234]{Brezis} and \cite[Theorem 9.11]{GilbargTrudinger} respectively) we will use.
We employ the notation $A \Subset B$ which means $\bar A$ is compact and $\bar A \subset B$.
\begin{lem}[Local Sobolev interpolation]
\label{interpol}
Let $V\Subset V'\Subset X$ be relatively compact open subsets. Then, for every $F\in H^2(V')$,
\begin{equation}\label{multInterpol}
        \|F\|_{H^1(V)} \ll_{V, V'} \|F\|_{L^2(V')}^{1/2}\|F\|_{H^2(V')}^{1/2}.
\end{equation}
\end{lem}

\begin{lem}[Interior regularity estimate]  
\label{intreg} If
$V\Subset U\Subset X$ are relatively compact open subsets, $F\in L^2(U)$,
and $\Delta F\in L^2(U)$ in the distributional sense, then
$F\in H^2(V)$, and
\begin{equation}\label{interior}
        \|F\|_{H^2(V)}
        \ll_{V,U}
        \|\Delta F\|_{L^2(U)}
        +
        \|F\|_{L^2(U)}.
\end{equation}
\end{lem}

\subsection{The operators $L_{\epsilon}$.}
From now on, we omit the $i$ from $\chi_{\epsilon}^i$ and let $w$  mean either $w^1$ or $w^2$, and $\chi_{\epsilon}$ either $\chi^1_{\epsilon}$ or $\chi^2_{\epsilon}$. Further, as in \cite{P} and \cite{PR}, we assume for simplicity that there is one cusp (at infinity) because the generalisation to the multiple cusps is straightforward.
By the identification of cuspidal cohomology with cohomology with compact support, we can assume that $w$ is a compactly supported form. We denote its support by $K_w$. We recall the construction of \cite[Section 2]{P}. 

Let
$L^{2}\bigl(\Gamma\backslash\mathbb H,\overline{\chi_{\epsilon}}\bigr)$
be the space of $L^{2}$-functions $h$ transforming as
$h(\gamma z)=\overline{\chi_{\epsilon}}(\gamma)h(z)$ ($\gamma\in\Gamma$). For $\epsilon \in \mathbb  R,$ we consider the unitary operators
$
U_{\epsilon}:L^{2}(\Gamma\backslash\mathbb H)
\longrightarrow
L^{2}\bigl(\Gamma\backslash\mathbb H,\overline{\chi_{\epsilon}}\bigr)
$
given by
\begin{equation*}
\bigl(U_{\epsilon}h\bigr)(z)
=
e^{ i\epsilon W(z)}h(z) \quad \text{where $W(z)=2 \pi \int_{\infty}^z w.$}
\end{equation*}
We set
\begin{equation*}
L_{\epsilon}
=
U_{\epsilon}^{-1}\Delta U_{\epsilon}.
\end{equation*}
This operator coincides with $\Delta$ outside $\supp(w)$. Further, the operators $L_{\epsilon}$ on $L^{2}(\Gamma\backslash\mathbb H)$ and
$\Delta$ on $L^{2}\bigl(\Gamma\backslash\mathbb H,\overline{\chi_{\epsilon}}\bigr)$
are unitarily equivalent. In terms of  the hyperbolic metric, a neighborhood of the cusp $\infty$  is isometric to
$[b,\infty)\times\mathbb R/\mathbb Z,$
for some $b>0$. We choose $b$ so that $\supp(w)$ does not intersect this neighborhood.

\subsection{The pseudo-Laplacian}
We next define pseudo-Laplacian operators associated with $L_{\epsilon}$ as in \cite[Section 2]{P}. For $A \ge a+1 > b+2,$ set
\begin{equation}\label{HA} H_{A}=\left\{
f\in L^2(X): f_{0}(y)=0, \, \, \text{for $y>A$}
\right\},
\end{equation}
where $f_{0}$ is the zero Fourier coefficient at infinity. Let $L_{\epsilon,A}$ be the Friedrichs extension in $H_A$
of the restriction to $H_{A} \cap H^1(X)$ of the quadratic form
$q(f)=\left\|\nabla U_{\epsilon}f\right\|^{2}$. The pseudo-Laplacian $L_{\epsilon,A}$ does not affect the non-zero Fourier coefficients but it removes the $0$-th Fourier coefficient for $y > A$. Thus $L_{\epsilon,A}$ acts as $L_\epsilon$ when $y<A$.
The operators $L_{\epsilon,A}$ have compact resolvents 
$$
R_{\epsilon, A}(s)
=
\bigl(L_{\epsilon,A}-s(1-s)\bigr)^{-1}
$$
that depend meromorphically on $s$ and real analytically on $\epsilon$. 

For the purpose of locating the potential poles of $Z^*_{m,n}$ we need to study more carefully the eigenvalues of $L_{\epsilon, A}$. We will show the following analog of \cite[Lemma 4.1]{P}. 
\begin{lem}\label{L4.1} Let $s_0$ with $\Re s_0>\frac12$. Suppose that $s_0(1-s_0)$ is not a cuspidal eigenvalue of $\Delta$ and $s_0$ is not a pole of $\phi(s)$. Then, $A$ may be chosen sufficiently large so that $R_{\epsilon,A}(s)$ is holomorphic in some neighborhood of $s_0$ and real-analytic in some neighborhood of $\epsilon=0$. Further,  $\partial_\epsilon R_{\epsilon,A}(s)|_{\epsilon=0}$ is holomorphic in some neighborhood of $s_0.$
\end{lem}
\begin{proof}  The proof of \cite[Lemma 4.1]{P} shows that under the given conditions, 
$A$ may be chosen large enough for $L_{0,A}$ to have no eigenvalue $s(1-s)$ with $s$ in a neighbourhood of $s_0$.  This also holds for $L_{\epsilon,A}$, for $|\epsilon|$ sufficiently small because $L_{\epsilon,A}$ is real-analytic and its spectrum is discrete (since it has compact resolvent).  Hence $R_{\epsilon,A}(s)$
is holomorphic near $s_0$ and real-analytic near $\epsilon=0$. This implies that $\partial_\epsilon R_{\epsilon,A}(s)|_{\epsilon=0}$ is also
holomorphic near $s_0$.
\end{proof}

\subsection{A function induced by $P_m(z, s)\coloneqq y^{1/2}I_{s-\frac12}(2\pi |m|y)e^{2\pi imx}.$ }\label{function}
We will construct a function which will turn out to coincide with a known Poincar\'e-like series. In each of the following subsections we will define an intermediate function and we will determine the location of its potential poles and its bound. All this information will be needed for the main theorem.

We follow the method of \cite[Section 2]{P}. For $a>b+1$, let
$h_a\in C^{\infty}(\mathbb R_{>0})$ such that $$h_a(y)=0, \, \, \text{for $y \leq a$ and $h_a(y)=1$ for $y\geq a+1$.}$$
Define $p_{m, s}$ on $ \Gamma_\infty\backslash\{x+iy; \ y>b+1\}$ by
$$p_{m, s}(x+iy)=h_a(y)P_m(z, s)=h_a(y)y^{1/2}I_{s-\frac12}(2\pi |m|y)e^{2\pi imx}.$$
By the periodicity of $e^{2\pi imx}$ this formula is well-defined close to $\infty$ in $X$. We extend by zero to the rest of $X$.
Since $h_a(y)=0$ for $y\leq a$ and $a>b+1$ this extension is smooth.

The choice of $P_m(z,s)$ is motivated by the analogy with the construction of \cite{P}. There, the objective was to study a perturbation of the Eisenstein series which has Fourier coefficients involving $S^*(m, 0; c)$. The construction was built on the ``seed" $y^s$ of the Eisenstein series. In our case, we would like a perturbed family of an analogous series with Fourier coefficients involving $S^*(m, n; c)$ for $n \neq 0$. A natural choice is the Poincar\'e series, but, as explained in \cite{GL}, a variation of the Poincar\'e series which is twisted by Bessel functions works better. This series has $P_m(z,s)$ as its ``seed" and we apply the method of \cite{P} to that function.

\subsubsection{$H_m(z, s)$} Since $W=0$ in $\supp(p_{m,s})$, we have $L_\epsilon p_{m,s}=\Delta p_{m,s}.$ In view of this, we define
\begin{equation}\label{H}H_{m}(\cdot,s)\coloneqq -(L_\epsilon-s(1-s))p_{m,s}=-(\Delta-s(1-s))p_{m,s}=y^2(2 h'_a \partial_y P_{m, s}+h''_a P_{m, s}).\end{equation}
This function is compactly supported on $X$ because for $y\geq a+1$, $L_\epsilon=\Delta$, $h_a(y)=1$ and, by a direct computation, $p_{m, s}$ is an eigenfunction of $\Delta$ with eigenvalue $s(1-s)$.  In addition, $p_{m,s}=0$ for $y \le a$. Hence $H_{m}(\cdot,s)\in L^2(X)$ since it is supported in $a<y<a+1$. In particular, $H_{m}(\cdot,s)=0$ for $y>A,$ and hence it belongs to $H_A$ of \eqref{HA}. 

To bound $\|H_m(\cdot, s)\|_{L^{2}(X)}$, we combine the bound $I_{s-\frac12}(x) \ll (x/2)^{\sigma-\frac12}/|\Gamma(s+\frac12)|$ with Stirling, to deduce that
$P_m(z, s) \ll_{m, a} e^{\pi |t|/2}(1+|t|)^{-1/2}$, {uniformly for $\frac12<\sigma\le1$ and $|t|>1$.} Further, with $I'_{s-1/2}(x)=I_{s+1/2}(x)+(s-1/2)I_{s-1/2}(x)/x$, we deduce $\partial_y P_m(z, s) \ll_{m, a} e^{\pi |t|/2}(1+|t|)^{1/2}$. With \eqref{H}, we deduce
\begin{equation}\label{boundH}\| H_{m}(\cdot,s)\|_{L^2(X)} \ll_{m, a} e^{\pi |t|/2}(1+|t|)^{1/2}.\end{equation}

\subsubsection{$R_{0, A}(s)H_{m}(\cdot,s)$}\label{RH} Since $R_{\epsilon,A}(s)$ meromorphically maps $H_A$ to itself, we deduce that, away from poles, \begin{equation}\label{R}R_{\epsilon,A}(s)H_{m}(\cdot,s)\in H_A \subset L^2(X).\end{equation}

Assume that $s_0$ with $\Re s_0>1/2$ is such that $s_0(1-s_0)$ is not a cuspidal eigenvalue of $\Delta$ and that $s_0$ is not a pole of $\phi(s)$. Since $H_{m}(\cdot,s)$ is holomorphic in $s$, Lemma \ref{L4.1} implies that, for some $A$ large enough, $R_{0,A}(s)H_{m}(\cdot,s)$ and $\partial_{\epsilon} R_{\epsilon,A}(s)H_{m}(\cdot,s) |_{\epsilon=0}$ are holomorphic in a neighborhood $D$ of $s_0$.

We will bound $R_{0, A}(s)H_{m}(\cdot,s)$ in terms of the $L^2$ and the $H^2$ norms. By the self-adjointness of $L_{0, A}$, the norm of its resolvent $R_{0, A}$ satisfies
$\| R_{0, A}\| \le |\Im(s(1-s))|^{-1} \ll 1/|t(2\sigma-1)|$ and hence, with \eqref{boundH}
\begin{equation}\label{boundRHL}\|R_{0, A}(s)H_{m}(\cdot,s)\|_{L^2(X)} \ll_m e^{\pi |t|/2}(1+|t|)^{-1/2} (2\sigma-1)^{-1}.\end{equation}
Let $K \Subset \{y<a\}$ or $K \Subset \{a+1<y<A\}.$ Then, in $K$, $H_m(\cdot, s)=0$ and $L_{0, A}=L_0=\Delta$. Hence the equation $(L_{0, A}-s(1-s))(R_{0, A}(s)H_{m}(\cdot,s))=H_m(\cdot, s)$ becomes
$$(\Delta-s(1-s))(R_{0, A}(s)H_{m}(\cdot,s))=0.$$ By Lemma \ref{intreg}, followed by an application of \eqref{boundRHL} this implies that
\begin{equation}\label{boundRHH}\|R_{0, A}(s)H_{m}(\cdot,s)\|_{H^2(K)} \ll_{m, K} (1+|s(1-s){|})\|R_{0, A}(s)H_{m}(\cdot,s)\|_{L^2(X)} \ll_{m, K} e^{\frac{\pi |t|}{2}}|t|^{\frac32} (2\sigma-1)^{-1}.
\end{equation}

\subsubsection{$F_{m, \epsilon,A}(z,s)$}
We now set
\begin{equation}\label{F}F_{m, \epsilon,A}(z,s)
        \coloneqq p_{m,s}(z)+R_{\epsilon, A}(s)H_{m}(\cdot, s)(z).\end{equation}
By the meromorphicity of $p_{m, s}, H_m(\cdot, s), R_{\epsilon, A}$ and the real-analyticity of $R_{\epsilon, A}$,  
we see that $F_{m, \epsilon,A}(z,s)$ is meromorphic in $s$ and real-analytic in
$\epsilon$.  Further, as mentioned above, $L_{\epsilon,A}$ agrees with $L_\epsilon$ for $y<A$, and hence, on such a region, 
$$  (L_\epsilon-s(1-s))F_{m, \epsilon ,A}
        =
        (L_\epsilon-s(1-s))p_{m,s}
        +(L_\epsilon-s(1-s))R_{\epsilon, A}(s)H_{m}(\cdot, s)  \\
        =
        -H_{m}+H_{m}=0.$$
This vanishing does not extend to all of $X$, essentially because, by working on $H_A$ we have ignored the $0$-th Fourier coefficient for $y>A$. We will therefore, modify $F_{m, \epsilon, A}$ so that we obtain a function that does satisfy the vanishing on $X$, while maintaining the other conditions of $F_{m, \epsilon, A}$ we require.

\subsubsection{$\alpha_{m,\epsilon}(s)$, $\beta_{m,\epsilon}(s)$}
First consider the zero Fourier coefficient
\begin{equation}\label{0thterm}F_{0,m,\epsilon, A}(y,s)
    \coloneqq         \int_0^1F_{m,\epsilon,A}(x+iy,s)\,dx= \int_0^1R_{\epsilon, A}(s)H_{m}(\cdot, s)(x+iy)\,dx\end{equation}
(the last equality follows because $p_{m,s}(z)$ is a multiple of $e^{2 \pi i m x}$ for $m \neq 0$.) For $b<y<A$, it is easy to see that $F_{0,m,\epsilon, A}(y,s)$ satisfies the differential equation $-y^2G''(y)= s(1-s) G(y).$
Hence, for $s\neq \frac12$, there exist meromorphic functions $\alpha_{m,\epsilon}(s)$, $\beta_{m,\epsilon}(s)$ such that
\begin{equation}\label{F0} F_{0,m,\epsilon,A}(y,s)
        =
        \alpha_{m,\epsilon}(s)y^s
        +
        \beta_{m,\epsilon}(s)y^{1-s},
        \qquad b<y<A. \end{equation}

We now verify the holomorphicity of $\alpha_{m, 0}(s), \beta_{m, 0}(s), \partial_{\epsilon} \alpha_{m, \epsilon}(s) |_{\epsilon=0}, \partial_{\epsilon} \beta_{m, \epsilon}(s) |_{\epsilon=0}$ in a neighborhood of $s_0$ such that $\Re s_0>1/2$, $s_0(1-s_0)$ is not a cuspidal eigenvalue of $\Delta$ and $s_0$ is not a pole of $\phi(s)$. Indeed, if $y_1 \neq y_2 \in(b,A)$, $\alpha_{m,\epsilon}(s), \beta_{m,\epsilon}(s)$  can be expressed as linear combinations of $F_{0,m, \epsilon,A}(y_1,s)$, $F_{0,m, \epsilon,A}(y_2,s)$ upon solving a system with determinant $y_1^s y_2^{1-s}-y_2^s y_1^{1-s} \neq 0$. By Lemma \ref{L4.1}, there is a $A$ large enough and a neighborhood $D$ of $s_0$ such that $R_{\epsilon,A}(s)$ is holomorphic in $s \in D$ and real-analytic in a neighborhood of $\epsilon=0$. This and Lemma \ref{intreg} imply that $F_{m,\epsilon,A}(z,s)$ is holomorphic as $H^2$-function. The same holds for $F_{0,m, \epsilon,A}$ and, by the Sobolev embedding $H^2 \hookrightarrow C^1$, we deduce that it is holomorphic. Thus $\alpha_{m,\epsilon}(s), \beta_{m,\epsilon}(s)$ are holomorphic in $s\in D$ and real-analytic near $\epsilon$.
This, further implies that $\partial_{\epsilon}\alpha_{m,\epsilon}(s) |_{\epsilon=0}, \partial_{\epsilon}\beta_{m,\epsilon}(s) |_{\epsilon=0}$ are holomorphic in $s\in D$.

We next bound $|\alpha_{m, 0}(s)|, |\beta_{m, 0}(s)|$. We consider an interval $I \Subset (a+1, A)$. Then, from \eqref{boundRHH}, we obtain
\begin{equation}\label{RH1}
\|R_{0, A}(s)H_{m}(\cdot, s)\|_{H^{2}([0, 1] \times I)} \ll e^{\pi |t|/2}(1+|t|)^{3/2} (2\sigma-1)^{-1}.
\end{equation} Combining this with \eqref{boundRHL} and Lemma \ref{interpol},
we deduce that
$$\|R_{0, A}(s)H_{m}(\cdot, s)\|_{H^{1}([0, 1] \times I)} \ll e^{\pi |t|/2}(1+|t|)^{1/2} (2\sigma-1)^{-1}.$$ Using \eqref{0thterm}, this and \eqref{RH1} imply
$$\|F_{0,m,0,A}(\cdot ,s)\|_{H^{1}(I)} \ll e^{\pi |t|/2}(1+|t|)^{1/2} (2\sigma-1)^{-1} \qquad  \text{and} $$
$$\|F_{0,m,0,A}(\cdot ,s)\|_{H^{2}(I)} \ll e^{\pi |t|/2}(1+|t|)^{3/2} (2\sigma-1)^{-1}.$$
With Sobolev embeddings  $H^1(I)\hookrightarrow C^0(\overline I)$ and $H^2(I)\hookrightarrow C^1(\overline I)$, we deduce that, for $y_0 \in I,$
\begin{equation}\label{boundF0}|F_{0,m,0,A}(y_0,s)| \ll e^{\pi |t|/2}(1+|t|)^{1/2} (2\sigma-1)^{-1}, \, \, \text{and} \, |F'_{0,m, 0,A}(y_0,s)| \ll e^{\pi |t|/2}(1+|t|)^{3/2} (2\sigma-1)^{-1}.\end{equation} 
From \eqref{F0} and the corresponding equation for $F'_{0, m, 0, A}(y, s)$ we obtain a system in $\alpha_{m, 0}(s)$ and $\beta_{m, 0}(s)$ with determinant $1-2s$ (and hence $|1-2s| \sim 1+|t|,$ for $|t|>1$). Solving it we deduce, with \eqref{boundF0}, 
\begin{equation}\label{boundalpha}|\alpha_{m,0}(s)|+|\beta_{m, 0}(s)| \ll e^{\pi |t|/2}(1+|t|)^{1/2} (2\sigma-1)^{-1}.\end{equation} 

\subsubsection{$\hat F_{m,\epsilon,A}(z,s)$}
Since $R_{\epsilon, A}(s)H_{m}(\cdot, s) \in H_A$, its zero Fourier
coefficient vanishes for $y>A$.  Since $p_{m,s}$ has no zero
Fourier coefficient, the zero Fourier coefficient of
$F_{m,\epsilon,A}$ is zero above height $A$.
Set
\begin{equation}\label{hat}\hat F_{m,\epsilon,A}(z,s)\coloneqq 
        F_{m,\epsilon,A}(z,s)
        +
        \mathbf 1_{[A,\infty)}(y)
        \left(
        \alpha_{m,\epsilon}(s)y^s
        +
        \beta_{m,\epsilon}(s)y^{1-s}
        \right)\end{equation}
in $[b,\infty)\times\mathbb R/\mathbb Z$, and $\hat F_{m,\epsilon,A}(z,s)\coloneqq F_{m,\epsilon,A}(z, s)$ elsewhere. 
The nonzero Fourier coefficients of $\hat F_{m,\epsilon,A}$ are equal to those of $F_{m,\epsilon,A}$.  On the other hand, the zero Fourier coefficient of
$F_{m,\epsilon,A}$ is zero for $y > A$, while, for $y<A$, it is $\alpha_{m,\epsilon}(s)y^s+\beta_{m,\epsilon}(s)y^{1-s}$. Therefore, adding that term when $y\ge A$ ensures that the  zero Fourier coefficient of $\hat F_{m, \epsilon, A}$ is smooth across $y=A$. Since, in addition, each term in the Fourier expansion of $\hat F_{m, \epsilon, A}$ satisfies $(L_\epsilon-s(1-s))g=0$, we have, distributionally in all of $X$,
\begin{equation}\label{eigenhat}        (L_\epsilon-s(1-s))\hat F_{m,\epsilon,A}=0.  \end{equation}
By Lemma \ref{intreg}, we obtain $U_{\epsilon}\hat F_{m,\epsilon,A} \in H^2$ locally.  The extension of Lemma \ref{intreg} to higher orders implies $U_{\epsilon}\hat F_{m,\epsilon,A} \in H^r$ for every $r$.  Hence, by Sobolev embedding, this function (and thus $\widehat F_{m,\epsilon,A}$) is smooth. Therefore, \eqref{eigenhat} holds classically.

\subsubsection{$D_{\epsilon}(z,s)$}
We next adjust $\hat F_{m, \epsilon, A}$ by a $y^s$-factor to ensure square-integrability.  Lemma 2.1 of \cite{P} gives, for $s$ in the resolvent set, a function $D_\epsilon(z,s)$ satisfying
\begin{equation}\label{Depsilon}(L_\epsilon-s(1-s))D_\epsilon(z,s)=0, \, \, \text{and $D_\epsilon(z,s)-h_{b+1}(y)y^s\in L^2(X).$}\end{equation} (However, note the different sign convention for $\Delta$ in \cite{P}, leading to $+s(1-s)$ instead of $-s(1-s)$.)
Further, with \cite[(2.21)-(2.23)]{P}, $D_\epsilon(z,s)$ is meromorphically continued to the $s$-plane, real-analytically in $\epsilon$.  Since $h_{b+1}-h_a$ is compactly supported,
$D_\epsilon(z,s)-h_a(y)y^s\in L^2(X).$

We bound $D_0(\cdot, s)$. Set
$H(z, s)=-(\Delta-s(1-s))(h_a(y)y^s).$ By the proof of \cite[Lemma 2.1]{P} we have $D_0(z, s)=h_a(y)y^s+R_0(s)H(z, s)$ and hence on $K_w$, $D_0(z, s)=R_0(s)H(z, s)$. 

Since, as in \eqref{H}, $H(z, s)=y^2(2h'_a(y) \partial_y y^s+h''_a(y) y^s)$, we deduce $\|H( \cdot, s)\|_{L^2(X)} \ll |t|.$
By the bound $\|R_0(s)\| \ll (|t|(2\sigma-1))^{-1},$ we deduce
$\|R_0(s)H(\cdot, s)\|_{L^2(X)} \ll (2\sigma-1)^{-1}$.
Therefore, with Lemma \ref{intreg},  
we obtain
\begin{equation}\label{boundD}\|D_0(\cdot, s)\|_{H^1(K_w)} \le \|D_0(\cdot, s)\|_{H^2(K_w)}  \ll (|t|+1)^2 \|R_0(s)H(\cdot, s)\|_{L^2(X)}\ll (|t|+1)^2(2\sigma-1)^{-1}.  \end{equation}

\subsubsection{$\tilde F_{m,\epsilon}(z,s)$}
The $0$-th Fourier coefficient of $D_{\epsilon}$ is $y^s+\phi_{\epsilon}(s)y^{1-s}$, for a scattering matrix $\phi_{\epsilon}(s)$ (see \cite[(5.4)]{P}).
Therefore, by \eqref{eigenhat} and \eqref{Depsilon},
\begin{equation}\label{tilde}\tilde F_{m,\epsilon}(z,s)
        \coloneqq 
        \hat F_{m,\epsilon,A}(z,s)
        -
        \alpha_{m,\epsilon}(s)D_\epsilon(z,s)\end{equation}
satisfies $(L_\epsilon-s(1-s))\tilde F_{m,\epsilon}=0$ and its zero Fourier coefficient at infinity for $y > A$ is $\left(
        \beta_{m,\epsilon}(s)
        -
        \alpha_{m,\epsilon}(s)\varphi_\epsilon(s)
        \right)y^{1-s}$
which is square-integrable for $\Re s>\frac12$.  With \eqref{R} and \eqref{Depsilon} we have, for $\Re s>\frac12$, 
\begin{multline}\label{tildeFL2}
        \tilde F_{m,\epsilon}(z,s)-p_{m,s}(z) 
=R_{\epsilon, A}(s)H_{m}(\cdot, s)(z)- \alpha_{m,\epsilon}(s) (D_{\epsilon}(z, s)-h_a(y)y^s)\\
+\left ( \mathbf 1_{[A,\infty)}(y)
        \left(
        \alpha_{m,\epsilon}(s)y^s
        +
        \beta_{m,\epsilon}(s)y^{1-s}\right )-  \alpha_{m,\epsilon}(s) h_a(y)y^s \right )
\in L^2(X).
\end{multline}

To bound $\tilde F_{m, 0}$, we first note that in $K_w$, we have $p_{m, s}(y)=0$ and $\mathbf 1_{[A, \infty)}(y)=0$, so
$\tilde F_{m, 0}(z, s)=R_{0,{ A}}(s) H_m(z, s)- \alpha_{m , 0}(s) D_0(z, s).$ Therefore, with \eqref{boundRHH}, \eqref{boundalpha} and \eqref{boundD}, we deduce
\begin{equation}\label{boundtilde}\|\tilde F_{m, 0}(\cdot ,s) \|_{H^1(K_w)} \ll e^{\pi |t|/2}(1+|t|)^{5/2}(2 \sigma-1)^{-2}. \end{equation}

\subsubsection{$U_\epsilon\tilde F_{m,\epsilon}(z,s)$}
By construction, $U_\epsilon\tilde F_{m,\epsilon}(z,s)$ is $\bar \chi_\epsilon$-automorphic, satisfies
\begin{equation}\label{PLem2.1}(\Delta-s(1-s))U_\epsilon\tilde F_{m,\epsilon}=0 \quad \text{and $U_\epsilon\tilde F_{m,\epsilon}-y^{1/2}I_{s-\frac12}(2\pi |m|y)e^{2\pi imx} \in L^2(X).$}\end{equation}
The last containment follows from \eqref{tildeFL2} because $p_{m, s}(z)=y^{1/2}I_{s-\frac12}(2\pi |m|y)e^{2\pi imx}$ and $U_{\epsilon}=1$ when $y \ge a+1$.

We will bound $F'_{m}\coloneqq \partial_{\epsilon} \tilde F_{m, \epsilon}|_{\epsilon=0}.$ From \eqref{PLem2.1} and the independence of $p_{m, s}$ from $\epsilon$, we have $F'_{m} \in L^2(X)$ and $(L_{\epsilon}-s(1-s))\tilde F_{m ,\epsilon}=0$. Upon differentiation, we deduce
\begin{equation}\label{Delta'}(\Delta-s(1-s))F'_m=-\partial_{\epsilon}L_{\epsilon}|_{\epsilon=0} \tilde F_{m, 0}. \end{equation}
By definition, $\partial_{\epsilon}L_{\epsilon}|_{\epsilon=0}$ is a first order differential operator and is supported in $K_w$.  Hence, with \eqref{boundtilde}
$$\|\partial_{\epsilon}L_{\epsilon}|_{\epsilon=0}  \tilde F_{m, 0} \|_{L^2(X)} \ll \| \tilde F_{m, 0}\|_{H^1(K_w)} \ll e^{\pi |t|/2}(1+|t|)^{5/2}(2 \sigma-1)^{-2}.$$
Since $\partial_{\epsilon}\tilde F_{m, \epsilon}|_{\epsilon=0} \in L^2(X)$, we can use \eqref{Delta'} and $\|R_0(s)\| \ll (|t|(2\sigma-1))^{-1}$ to obtain, away from the poles of the resolvent,
\begin{equation}\label{boundF} \|F'_m \|_{L^2(X)}=  \|-R_0(s)\partial_{\epsilon} L_{\epsilon} |_{\epsilon=0} \tilde F_{m, 0} \|_{L^2(X)} \ll e^{\pi |t|/2}(1+|t|)^{3/2}(2 \sigma-1)^{-3}.\end{equation}

\subsection{The proof of the meromorphic continuation of $Z^*_{m, n}(s)$}
For $\Re s>1$, Niebur's Poincar\'e series 
$$F_m(z,s,\chi_\epsilon)
        =
        \sum_{\gamma\in\Gamma_\infty\backslash\Gamma}
        \chi_\epsilon(\gamma)
        \Im(\gamma z)^{1/2}
        I_{s-\frac12}(2\pi |m|\Im(\gamma z))
        e^{2\pi im\Re(\gamma z)}
$$
is absolutely convergent and satisfies the conditions \eqref{PLem2.1} (see \cite[Section 3]{F}, but note the different normalisation employed there).  Therefore, $U_\epsilon\tilde F_{m,\epsilon}(z,s)-F_m(z,s,\chi_\epsilon)$
is an $L^2$-eigenfunction of $\Delta$ with eigenvalue $s(1-s)$.  Since $\Re s>1$, by \cite[Satz 5.5]{Roe} this eigenvalue is not in the
$L^2$-spectrum.  Hence, the difference vanishes and we have
    \begin{equation}\label{key}    U_\epsilon\tilde F_{m,\epsilon}(z,s)= F_m(z,s,\chi_\epsilon),
        \qquad \text{for $\Re s>1$}.\end{equation}
It follows that the nonzero Fourier coefficients of $F_m(z,s,\chi_\epsilon)$
are, for $\Re(s)>1$, meromorphic in
$s$ and real-analytic in $\epsilon$, away from a polar set.  Therefore, \eqref{key} gives the meromorphic continuation in $s$ and real-analytic dependence on $\epsilon$
of these Fourier coefficients. 

In \cite[Theorem 3.4]{F}, the Fourier expansion of $F_{n}(z,s,\chi_\epsilon),$ is computed. In the specific case we require (${m>0>n}, k=\kappa=0, \Gamma=\Gamma_0(N)$), the coefficient of ${e^{2\pi imx}}$ in this expansion is 
\begin{equation}\label{FourierFm}
\phi_{m, n}(s; \chi_{\epsilon}, y)\coloneqq
2\sqrt{y}\sum_{N|c>0} \frac{S_{\chi_{\epsilon}}(m, n; c)}{c}I_{2s-1}
\left ( \frac{4 \pi \sqrt{{|}mn{|}}}{c}\right )K_{s-\frac12}(2 \pi {m} y).
\end{equation}
Therefore, since the only zeros in $w$ of $K_w(x)$ for a fixed $x>0$ lie on the imaginary axis, we can divide, for $\Re(s)>\frac12$,  this Fourier coefficient with $2 \sqrt{y}K_{s-\frac12}(2 \pi {m} y)$ to deduce
\begin{prop}\label{realanal} The function 
$$\sum_{N|c>0} \frac{S_{\chi_{\epsilon}}(m, n; c)}{c}I_{2s-1} \left ( \frac{4 \pi \sqrt{{|}mn{|}}}{c}\right )$$
 has a meromorphic continuation to $\Re(s)>\frac12$ which is real-analytic in $\epsilon$.
\end{prop}
We will now express this Fourier coefficient in terms of $Z_{m, n}(s; \chi_\epsilon)$. 
By the series expression of the $I$-Bessel function, we have, for $\Re(s) \gg 1$,
\begin{equation*}
\begin{aligned}
\phi_{m, n}(s; \chi_\epsilon)
&\coloneqq \sum_{N|c>0} \frac{S_{\chi_{\epsilon}}(m, n; c)}{c}I_{2s-1}
\left ( \frac{4 \pi \sqrt{{|}mn{|}}}{c}\right )\\
&=\sum_{\ell \ge 0} \frac{(2 \pi \sqrt{{|}mn{|}})^{2s+2\ell-1}}{\ell! \Gamma(2s+\ell)} Z_{m, n }(s+\ell, \chi_\epsilon).
\end{aligned}
\end{equation*}
Hence
\begin{equation}\label{phimn}
\begin{aligned}
Z_{m, n }(s, \chi_\epsilon)
&=\frac{\Gamma(2s)}{(2 \pi \sqrt{{|}mn{|}})^{2s-1}}\Biggl( \phi_{m, n}(s; \chi_\epsilon)\\
&\qquad-\sum_{\ell \ge 1} \frac{(2 \pi \sqrt{{|}mn{|}})^{2s+2\ell-1}}{\ell! \Gamma(2s+\ell)} Z_{m, n }(s+\ell, \chi_\epsilon) \Biggr).
\end{aligned}
\end{equation}
\begin{lem}\label{tail} The function 
$$\sum_{\ell \ge 1}\frac{(2 \pi \sqrt{{|}mn{|}})^{2\ell}\Gamma(2s)}{\ell! \Gamma(2s+\ell)} Z_{m, n }(s+\ell, \chi_\epsilon)$$ is real-analytic in $\epsilon$ and holomorphic in $\Re(s)>\frac12$. 
\end{lem}
\begin{proof}
For $\Re(s)>\frac12$ and $\ell \ge 1$, $\Re(2s+2\ell)>3$ and hence $Z_{m, n }(s+\ell, \chi_\epsilon)$ converges absolutely and uniformly in $\epsilon$. Therefore, for $s$ in a compact subset $K$ of $\Re(s)>\frac12$, the series is 
$$\ll_{K} \sum_{\ell \ge 1} \frac{(2 \pi \sqrt{{|}mn{|}})^{2 \ell}}{\ell! |\Gamma(2s+\ell)|}.$$ The ratio test shows that this converges absolutely, thus deducing the assertion.
\end{proof}
We are now ready to show the following.
\begin{thm}\label{meromcontin} 
Suppose that $m,n\ne0$.
\begin{enumerate}
\item\label{mero-i} The zeta function $Z^*_{m, n}(s)$ has meromorphic continuation to the region $\Re(s)>1/2$.
\item\label{mero-ii} The only potential poles in the strip $\Re(s) \in (1/2, 1)$ are those $s$ such that $s(1-s)$ is an exceptional eigenvalue of $\Delta$. {Each such pole has order at most two.}
\item\label{mero-iii} For $|t|>1$ and $\frac12<\sigma  { < 1}$, we have
$$Z_{m, n}^*(s) \ll_{m, n, N, f, \varepsilon} \frac{|t|^{\frac72+\epsilon}}{(\sigma-\frac12)^3}.$$
\end{enumerate}
\end{thm}
\begin{proof} {Fix $i\in\{1,2\}$ and write $\chi_\epsilon=\chi^i_\epsilon$ throughout the proof. Put $\dot Z^{\,i}_{m,n}(s)\coloneqq\left.\partial_\epsilon Z_{m,n}(s;\chi^i_\epsilon)\right|_{\epsilon=0}.$ We prove the claims for each real direction and combine them using $Z^*_{m,n}=\dot Z^{\,1}_{m,n}+i\dot Z^{\,2}_{m,n}$, as in \eqref{Z*-Zchi}.} We treat the case {$m>0>n$} for convenience.  For other signs, we must use the absolute values of $m$, $n$ below,
and if $\sgn(mn)={1},$ then the corresponding Fourier coefficients of the
Niebur Poincar\'e series involve the $J$-, rather than the $I$-Bessel function. 

For part~\ref{mero-i}, combining Proposition \ref{realanal} and Lemma  \ref{tail}, we deduce that $Z_{m, n}(s; \chi_\epsilon)$ is real-analytic in $\epsilon$ and meromorphic for $\Re(s)>1/2$. Therefore, with \eqref{Z*-Zchi}, we deduce the meromorphic continuation of  $Z^*_{m, n}(s)$ in $\Re(s)>1/2$.

For part~\ref{mero-ii}, to show the potential location of the poles, we first note that, by \eqref{phimn} and the holomorphicity established in Lemma \ref{tail}, it is sufficient to show the holomorphicity of $\partial_{\epsilon} \phi_{m, n}(s; \chi_\epsilon) |_{\epsilon=0}$ in a neighborhood of each $s_0$ such that $1>\Re s_0>1/2$, $s_0(1-s_0)$ is not a cuspidal eigenvalue of $\Delta$ and $s_0$ is not a pole of $\phi(s)$.  By the construction of $\phi_{m, n}(s; \chi_\epsilon)$ from the Fourier coefficients of $F_{{n}, \epsilon}$ and \eqref{key}, we deduce that this can be deduced from the holomorphicity of $\partial_{\epsilon} U_{\epsilon}\tilde F_{{n}, \epsilon} |_{\epsilon=0}$ in a neighborhood of $s_0$. A direct differentiation shows that this is a linear combination of  $\tilde F_{{n}, 0}$ and
$\partial_{\epsilon} \tilde F_{{n}, \epsilon} |_{\epsilon=0}$.

Now, by \eqref{tilde}, the holomorphicity of $\tilde F_{{n}, 0}$ and
$\partial_{\epsilon} \tilde F_{{n}, \epsilon} |_{\epsilon=0}$ near $s_0$ is reduced to the holomorphicity of 
$\hat F_{{n}, 0 ,A}(z, s),$ $\alpha_{{n}, 0}(s)$,  $\beta_{{n}, 0}(s),$ $D_0(z,s),$ $\partial_{\epsilon}\hat F_{{n},\epsilon,A}(z, s)|_{\epsilon=0},$ 
$\partial_{\epsilon}\alpha_{{n},\epsilon}(s)|_{\epsilon=0},$ $\partial_{\epsilon}\beta_{{n},\epsilon}(s)|_{\epsilon=0}$ and $\partial_{\epsilon}D_\epsilon(z,s)|_{\epsilon=0}.
$
For $\alpha_{{n}, 0}(s),$  $\beta_{{n}, 0}(s)$, $\partial_{\epsilon}\alpha_{{n},\epsilon}(s)|_{\epsilon=0}$ and $\partial_{\epsilon}\beta_{{n},\epsilon}(s)|_{\epsilon=0},$ this has been proved in Section \ref{function}, immediately after the introduction of $\alpha_{{n}, \epsilon}(s)$ in \eqref{F0}. 
This, in turn, together with  \eqref{hat}, \eqref{F} and the holomorphicity of $R_{0,A}(s)H_{{n}}(\cdot,s)$ and $\partial_{\epsilon} R_{\epsilon,A}(s)H_{n}(\cdot,s) |_{\epsilon=0}$ shown in Section \ref{function} (right after \eqref{R}) establishes the holomorphicity of $\hat F_{{n}, 0, A}(z, s)$ and $\partial_{\epsilon}\hat F_{{n},\epsilon,A}(z, s)|_{\epsilon=0}$. The holomorphicity of $D_0(z, s)$ and $\partial_{\epsilon}D_\epsilon(z,s)|_{\epsilon=0}$ follows from the proof of \cite[Lemma 4.1]{P}. 

To determine pole orders, use the resolvent $R_0(s)=(\Delta-s(1-s))^{-1}$ appearing in \eqref{boundF}. From \eqref{H} and \eqref{tildeFL2}, uniqueness of the $L^2$ solution gives
\[
\tilde F_{n,0}-p_{n,s}=R_0(s)H_n(\cdot,s),
\qquad
F'_n=-R_0(s)V R_0(s)H_n(\cdot,s),
\]
initially for $\Re(s)>1$, and then by meromorphic continuation. Here $V=\left.\partial_\epsilon L_\epsilon\right|_{\epsilon=0}$, and the second identity follows from \eqref{Delta'} and $Vp_{n,s}=0$. In $\Re(s)>\frac12$, the full resolvent has only simple poles at the parameters of discrete eigenvalues below $1/4$. The displayed product therefore has poles of order at most two. Since $W(z)=0$ high in the infinity cusp, \eqref{key}, \eqref{FourierFm} and \eqref{phimn} imply this order bound for $\dot Z^{\,i}_{m,n}$.

To prove part~\ref{mero-iii}, from Lemma \ref{tail}, we have that 
$$\begin{aligned}
\partial_\epsilon \left ( \sum_{\ell \ge 1}\frac{(2 \pi \sqrt{{|}mn{|}})^{2\ell}\Gamma(2s)}{\ell! \Gamma(2s+\ell)} Z_{m, n }(s+\ell, \chi_\epsilon^{i})\right )\Big |_{\epsilon=0}=\sum_{\ell \ge 1}\frac{(2 \pi \sqrt{{|}mn{|}})^{2\ell}\Gamma(2s)}{\ell! \Gamma(2s+\ell)} {\dot Z^{\,i}_{m,n}(s+\ell)}.
\end{aligned}$$ Now, for $1>\Re(s)>1/2,$ $|t|>1$ and $\ell \ge 1$, ${\dot Z^{\,i}_{m,n}(s+\ell)}$ is absolutely convergent and bounded uniformly in $s, \ell$ in the stated range. Since $|\Gamma(2s)/\Gamma(2s+\ell)|=|(2s) \dots (2s+\ell-1)|^{-1} \le |t|^{-\ell}$, the sum is $\ll |t|^{-1} \ll 1.$ Therefore, by \eqref{phimn}, it suffices to bound
\begin{equation}\label{nthFourier} \frac{\Gamma(2s)}{(2 \pi \sqrt{{|}mn{|}})^{2s-1}} \partial_{\epsilon}\phi_{m, n}(s; \chi_\epsilon)|_{\epsilon=0}.
\end{equation}
Fix $Y>A+1$. Then, for $y>Y$, $U_{\epsilon}=1$ and hence, by \eqref{key} and \eqref{FourierFm}, the coefficient of ${e^{2\pi imx}}$ in the expansion of $F'_{n}(x+iy)$ is
\[
\partial_{\epsilon}\phi_{m,n}(s;\chi_{\epsilon},y)|_{\epsilon=0}
=2\sqrt{y}\,K_{s-\frac12}(2\pi {m}y)\partial_{\epsilon}\phi_{m,n}(s;\chi_{\epsilon})|_{\epsilon=0}.
\]
Further,
$$\|F'_{n}\|^2_{L^2(X)} \ge \int_{Y}^{Y+1} \int_0^1  |F'_{n}(x+iy)|^2 \frac{dx dy}{y^2}$$
and thus, with Parseval in the inner integral followed by the lower bound for the integral of $K$-Bessel functions of the proof of \cite[Theorem 3.2]{IwS}, we deduce
$$\|F'_{n}\|^2_{L^2(X)} \ge 4 |\partial_{\epsilon} \phi_{m, n}(s, \chi_\epsilon)|_{\epsilon=0}|^2 \int_{Y}^{Y+1}  |K_{s-\frac12}(2 \pi {m} y)|^2 \frac{dy}{y} \gg |\partial_{\epsilon} \phi_{m, n}(s, \chi_\epsilon)|_{\epsilon=0}|^2  e^{-\pi |t|} |t|^{-1}.$$
With \eqref{boundF}, this gives $ |\partial_{\epsilon} \phi_{m, n}(s, \chi_\epsilon)|_{\epsilon=0}| \ll e^{\pi |t|}|t|^2 (2 \sigma-1)^{-3}$. { With \eqref{nthFourier} and Stirling for $\Gamma(2s)$, the bound for $\dot Z^{\,i}_{m,n}$ follows. Applying it for $i=1,2$, \eqref{Z*-Zchi} implies the assertion for $Z^*_{m,n}$.}
\end{proof}

We will now use this theorem to establish the analogue of the second main result of \cite{GS}. 
\begin{thm}\label{Theorem2} { Suppose that $m,n\neq0$.} Let
$$\beta\coloneqq \overline{\lim_{\substack{c \to \infty \\ N|c}}}\frac{\log|S^*(m, n; c)|}{\log c}{.}$$
{ If $0<\beta<1$, there are $r_j,q_j\in\mathbb C$} and $\alpha_j \in (0, 1)$ such that, for all $\varepsilon>0$, 
$$\sum_{\substack{0<c<x \\ N|c}}\frac{S^*(m,n; c)}{c}=\sum_{j=1}^{M} {x^{\alpha_j}(q_j\log x+r_j)}+O_{m,n} \left (x^{\frac{7\beta}{9}+\varepsilon}\right ).$$
\end{thm}
\begin{proof} Thanks to the previous results, it is possible to follow very closely the proof of \cite[Theorem 2]{GS}. {Initially let $x$ be a half-integer, and choose an auxiliary $\varepsilon>0$ so that $\beta+\varepsilon<1$ and $\varepsilon<2s_j-1$ for every pole $s_j\in(1/2,1)$ of $Z^*_{m,n}$.} First, by \eqref{std-taub-int} and the absolute convergence of $Z^*_{m, n}((1+s)/2)$ on $\Re(s)=\beta+\varepsilon$,  $$\sum_{\substack{0<c<x \\ N|c}}\frac{S^*(m,n; c)}{c}=\frac{1}{2 \pi i} \int_{(\beta+\varepsilon)} Z^*_{m, n} \left ( \frac{1+s}{2} \right ) \frac{x^s}{s}ds.$$ With \cite[Lemma, Chapter 17]{D}, this becomes, for {$2\le T\le x$},
$$\sum_{\substack{0<c<x \\ N|c}}\frac{S^*(m,n; c)}{c}=\frac{1}{2 \pi i} \int_{\beta+\varepsilon-iT}^{\beta+\varepsilon+iT} Z^*_{m, n} \left ( \frac{1+s}{2} \right ) \frac{x^s}{s}ds+O \left ( \frac{x^{\beta+\varepsilon}}{T}\right ).$$ 
With Cauchy's Residue Theorem this becomes
\begin{multline}\label{cauchy}\sum_{\substack{0<c<x \\ N|c}}\frac{S^*(m,n; c)}{c}=\sum_{j=1}^M \underset{s=2 s_j-1}{\text{Res}} \left (Z^*_{m, n} \left ( \frac{s+1}{2}\right ) {\frac{x^s}{s}}\right ) \\
+\frac{1}{2 \pi i} \left ( \int_{\varepsilon-iT}^{\varepsilon+iT}+  \int_{\varepsilon+iT}^{\beta+\varepsilon+iT} - \int_{\varepsilon-iT}^{\beta+\varepsilon-iT} \right ) Z^*_{m, n} \left ( \frac{1+s}{2} \right ) \frac{x^s}{s}ds+O \left ( \frac{x^{\beta+\varepsilon}}{T}\right ).\end{multline}
{Here $s_1,\ldots,s_M$ are these poles, with $\alpha_j=2s_j-1\le\beta<1$ by absolute convergence. Each residue is $x^{\alpha_j}(q_j\log x+r_j)$, since each pole has order at most two.}
To bound the integrals in the RHS we use the following bound, which is deduced by Theorem \ref{meromcontin}(iii) and convexity{, after subtracting the finite principal parts}:
\begin{equation}\label{PhL} 
Z_{m, n}^* \left ( \frac{1+s}{2}\right ) \ll |t|^{\frac72-\frac{7{(\sigma-\varepsilon)}}{2\beta}+\varepsilon}, \qquad \text{if $\varepsilon \le \sigma \le \beta+\varepsilon$ {and $|t|\ge1$}.}
 \end{equation} 
We have
\begin{equation}\label{shifted} \int_{\varepsilon-iT}^{\varepsilon+iT}Z^*_{m, n} \left ( \frac{1+s}{2} \right ) \frac{x^s}{s}ds \ll x^{\varepsilon}{\left(1+\int_1^T t^{\frac52+\varepsilon}\,dt\right)} \ll x^{\varepsilon}T^{\frac72+\varepsilon} \, \, \text{and}\end{equation}
$$\int_{\varepsilon \pm iT}^{\beta+\varepsilon \pm iT} Z^*_{m, n} \left ( \frac{1+s}{2} \right ) \frac{x^s}{s}ds \ll T^{\frac72+\varepsilon{+\frac{7\varepsilon}{2\beta}}}\int_{\varepsilon}^{{\beta+\varepsilon}} \frac{(x/T^{7/(2\beta)})^{\sigma}d\sigma}{\sqrt{\sigma^2+T^2}}\le 
 T^{\frac72+\varepsilon{+\frac{7\varepsilon}{2\beta}}}\int_{{\varepsilon}}^{{\beta+\varepsilon}} \frac{(x/T^{7/(2\beta)})^{\sigma}d\sigma}{T}.$$
If $x>T^{7/(2\beta)},$ this is  $\le T^{\frac52+\varepsilon{+\frac{7\varepsilon}{2\beta}}}(x/T^{7/(2\beta)})^{\beta+\varepsilon} \le T^{-1{+\varepsilon}}x^{\beta+\varepsilon}$, otherwise, $\le T^{\frac52+\varepsilon}x^{\varepsilon}.$
Combining this with \eqref{shifted} and \eqref{cauchy} we deduce that, in either case,
$$\sum_{\substack{0<c<x \\ N|c}}\frac{S^*(m,n; c)}{c}=\sum_{j=1}^M \underset{s=2 s_j-1}{\text{Res}} \left (Z^*_{m, n} \left ( \frac{s+1}{2}\right ) {\frac{x^s}{s}}\right ) +O \left ( \frac{x^{\beta+\varepsilon}}{T^{1-\varepsilon}}+T^{\frac72+\varepsilon}x^{\varepsilon}\right ).$$
Selecting $T=x^{2\beta/9}/(7/2)^{2/9}$, we deduce the theorem {for half-integer $x$. Passing to arbitrary $x$ changes the sum and the main terms by $O(1)$, since $\beta+\varepsilon<1$ and $\alpha_j\le\beta<1$}.
\end{proof}

\section{The relation with exceptional eigenvalues}\label{Cool-numerical-section}

There is a deep connection between the exceptional spectrum on $\Gamma_0(N)$ and analytic properties of the standard Kloosterman zeta function, $Z_{m,n}(s)=\sum_{N\mid c>0}c^{-2s}S(m,n,c)$. In particular, if $Z_{m,n}(s)$ is holomorphic in $\Re(s)>\frac12$ {for all $m, n \neq 0,$} then the exceptional spectrum is empty.

From Theorem~\ref{meromcontin},
there is also a connection between the analytic properties of 
$Z^{*}_{m,n}(s)
$ and the existence
of exceptional eigenvalues. Despite both being related to the exceptional spectrum, $Z_{m,n}(s)$ and $Z^{*}_{m,n}(s)$ seem to be uncorrelated as we will see below.

We recall here that 
the 
Selberg eigenvalue conjecture asserts that 
there are no exceptional eigenvalues for congruence subgroups. This conjecture has been verified in special cases by a number of authors.
The most recent and extensive work is 
\cite{BSL}, where they show, for instance, that it holds for $\Gamma_1(N)$ with $N\le880$ and for $\Gamma(N)$ with  $N\le 226$.

In the remainder of this section we will present the results of some 
numerical investigations of $Z_{{m, n}}, Z^*_{m, n}$ mainly from the perspective of exceptional eigenvalues. 

\subsection{Selberg -- Linnik type bounds}\label{1stconjecture}

A key element in investigating the singularities of $Z^{*}_{m,n}(s)$ is the Selberg -- Linnik type bound for the partial sums at $s=1/2$. We conjecture
\begin{conj}\label{new-SL-conj} For $m,n\neq0$,
\begin{equation} 
Z_{m, n}^*(x)\coloneqq\sum_{N\mid c\le x}\frac{S^*(m,n;c)}{c}  \ll_{m,n,\varepsilon} x^\varepsilon.\label{linnikbd}
\end{equation}
\end{conj}
First observe that the weights,
$\langle f,\gamma_d\rangle$, are not of modulus one: suitably normalized modular symbols
have a normal distribution with variance $\sim \log c$, by Petridis--Risager
\cite{PR} 
so even the termwise (Weil-type) behaviour of $S^*$ is not
obvious. Second, as in Section \ref{6},  
the relevant aim
is the behaviour of $Z_{m,n}(s;\chi_\epsilon)$ near $\Re s = 1/2$
\emph{uniformly} in $\epsilon$. Then, the embedded eigenvalues of $\Gamma_0(N)$ move (and may
dissolve into resonances) under the character deformation, in the spirit of
Phillips--Sarnak \cite{PS, P}. Numerical evidence is
therefore of genuine interest.

Since this sum involves a large number of modular symbol evaluations  for large moduli we implemented an algorithm which computes a complete set of 
$\langle f, \gamma_d\rangle$ where  $\gamma_d=\left(\begin{smallmatrix}
	a &(ad - 1)/c \\
	c & d
\end{smallmatrix}\right)$ with $a^{-1}\equiv d\pmod{c}$ in terms of the 
Eichler integral of $f$ evaluated at all the points $(j+i)/c$, $j\bmod c$, simultaneously by a single FFT, at a cost of $O(c\log c)$ per modulus.
The complete runtime for evaluating $Z_{m, n}^{*}(x)$ is then easily seen to be 
$O(x^2\log x / N)$. Note that the computation of Fourier coefficients 
$a(n)$ is done once for each $N$ and, if using e.g. Schoof's algorithm, the cost is negligible compared to the main term above. 

The computed data cover all $28$ levels $N\le 50$ carrying a newform of weight $2$ ($84$ series with
$(m,n)\in\{(1,1),(1,2),(1,5)\}$ and $c\le 10^5$), levels $N=11,49,50$ to
$c\le 10^6$, and the CM level $N=49$ to $c\le 10^7$.
All data is consistent with the Linnik--Selberg--type bound \eqref{linnikbd}.

 In Figure \ref{fig:app-N11-1-1} we illustrate the case of $N=11$.  In the first figure we plot $Z_{m, n}^{*}(x)$ and $Z_{m, n}(x)$ for 
 $x \in [0,10^6]$. 
The middle figure shows the absolute value $|Z_{m, n}^{*}(x)|$ on a log-log scale 
together with a log-linear least square fit with slope $\alpha$, of the window-max envelope where maxima are taken over 24 logarithmically-spaced windows.     
The right-most figure shows the growth of the individual twisted Kloosterman sums
$|S^{*}(m,n,c)|$ on a log-log scale 
together with {$\sqrt{c}$} and window-max envelopes and 
the log-linear least square fit. 
\begin{figure}[H]\centering
	\includegraphics[width=\textwidth]{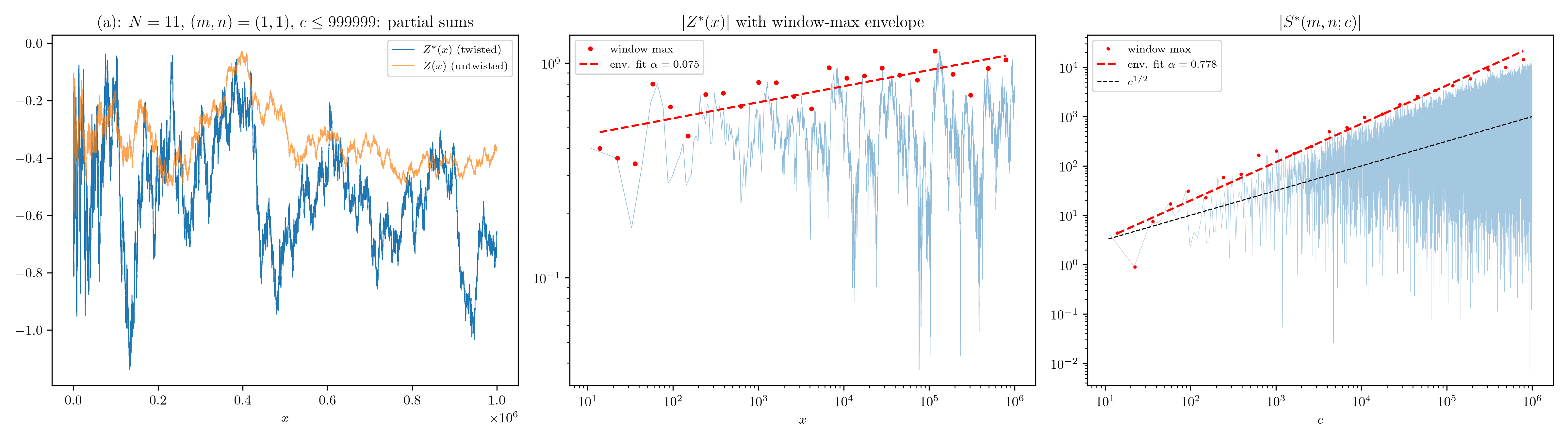}
	\caption{$N=11$, $(m,n)=(1,1)$, $c\le999999$.}
	\label{fig:app-N11-1-1}
\end{figure}
\begin{figure}[H]\centering
	\includegraphics[width=\textwidth]{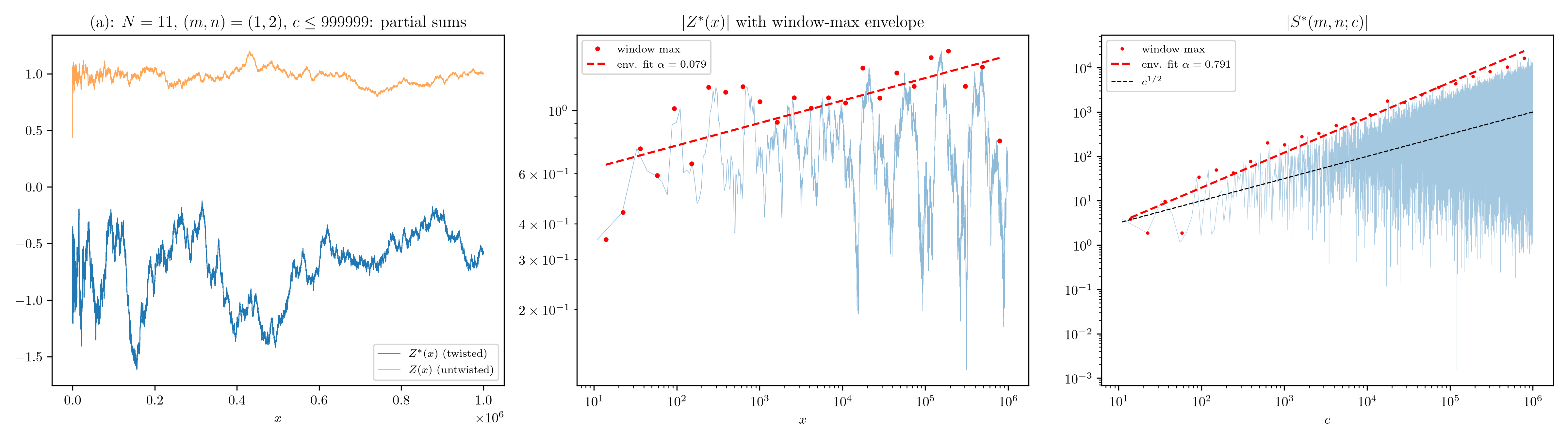}
	\caption{$N=11$, $(m,n)=(1,2)$, $c\le999999$.}
	\label{fig:app-N11-1-2}
\end{figure}
\begin{figure}[H]\centering
	\includegraphics[width=\textwidth]{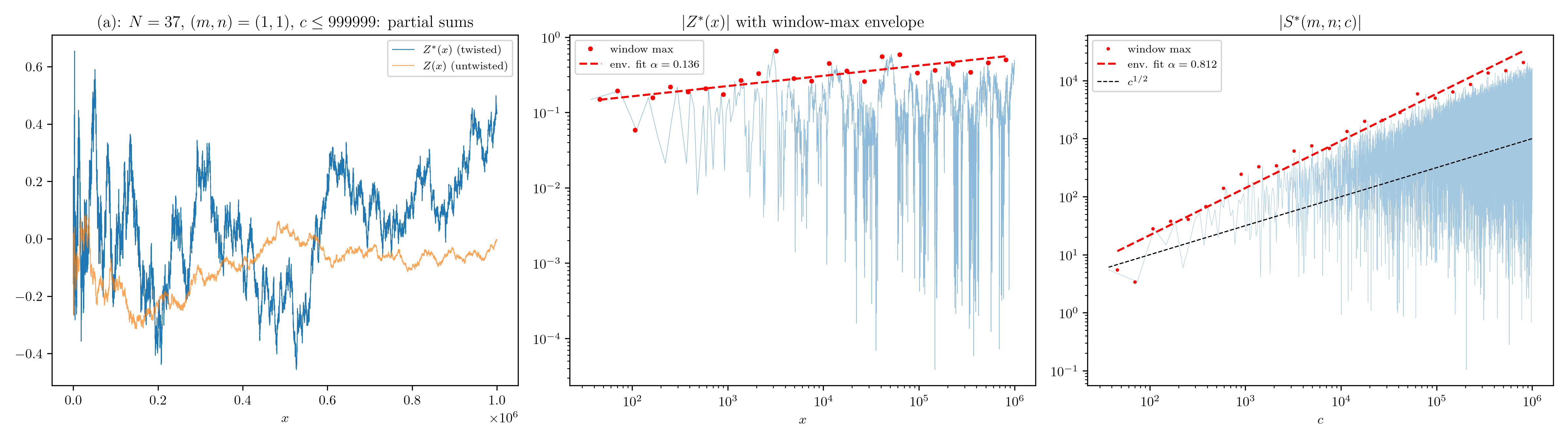}
	\caption{$N=37$, $(m,n)=(1,1)$, $c\le999999$.}
	\label{fig:app-N37-1-1}
\end{figure}
\begin{figure}[H]\centering
	\includegraphics[width=\textwidth]{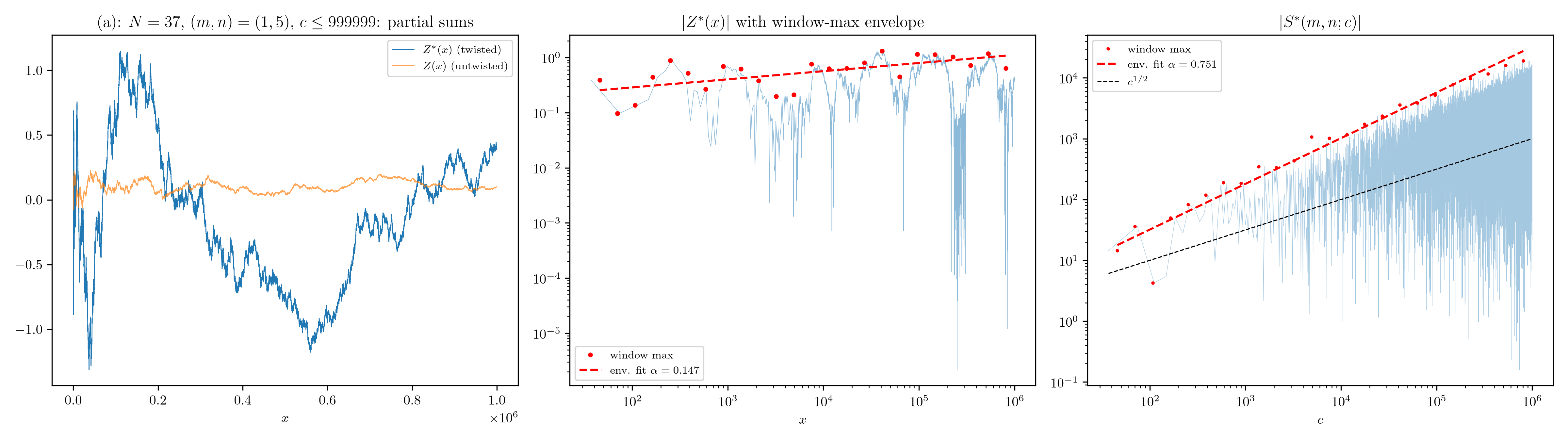}
	\caption{$N=37$, $(m,n)=(1,5)$, $c\le999999$.}
	\label{fig:app-N37-1-5}
\end{figure}

At squarefree levels $N$ the partial sums remain bounded as $c$ varies over 6 decades (i.e. intervals of the form $[10^k,10^{k+1})$,
closely tracking the ordinary Kloosterman partial sums, $Z(x)$, with centred fluctuations of
size $\sqrt{c}\log c$. The extra logarithmic factor was observed numerically.

These computations bear directly on the poles in Theorem~\ref{Theorem2}. A
block-convergence test indicates that $Z^*_{m,n}(s)$ has no pole in $\Re(s)>0.56$ at any
of the levels examined, {which is} consistent with the absence of exceptional eigenvalues there
 so that no main term ${x^{\alpha_j}(q_j\log x+r_j)}$ is detected and $Z^*_{m,n}(x)$ remains of size
$\log^2 x$ or smaller, far below the proved error $O\big(x^{7\beta/9+\varepsilon}\big)$.
The termwise exponent is measured to be $\beta=\tfrac12$, and the vertical growth of
$Z^*_{m,n}(s)$ in the strip $\tfrac12<\Re(s)<\tfrac34$ is found to be of essentially
bounded order, well below the convexity bound that yields the exponent $7\beta/9$; this
suggests that the error term in Theorem~\ref{Theorem2} is far from sharp and could be
improved under a Lindel\"of-type hypothesis for $Z^*_{m,n}$.

\subsection{Correlation of $S(m,n;c)$ and $S^*(m,n;c)$}

As $S(m,n;c)$ and $S^*(m,n;c)$ are both related to the exceptional spectrum 
and both involve a sum of exponential phases $e((na+md)c^{-1})$ 
with weights $1$ and $\langle f,\gamma_d\rangle$, respectively, 
 one might expect that they are correlated. 
However, the numerical evidence suggests otherwise. 

Define 
$$\rho_{m,n}(x) = \#\{c \le x; N \mid c, \operatorname{sign}S(m,n;c) = \operatorname{sign}S^*(m,n;c)\}
/\#\{c \le x; N \mid c\}.$$ Then uncorrelated signs would correspond to $\rho_{m,n}(x) \to 1/2$ as $x \to \infty$, 
in line with viewing the sequence as a Bernoulli process with probability $1/2$ (a sequence of fair coin tosses).
This behaviour is indeed observed in the numerical data, as illustrated in Figure \ref{fig:correlation}, where $\rho_{m,n}(x)$ is plotted over a sequence of dyadic windows in $c$ up to $10^6$ for the prime 
levels $N=11, 17, 19, 37, 43$ and three different pairs of $(m,n)$. 
The shaded regions are representing two standard deviations.

\begin{figure}
        \centering
        \includegraphics[width=0.8\textwidth]{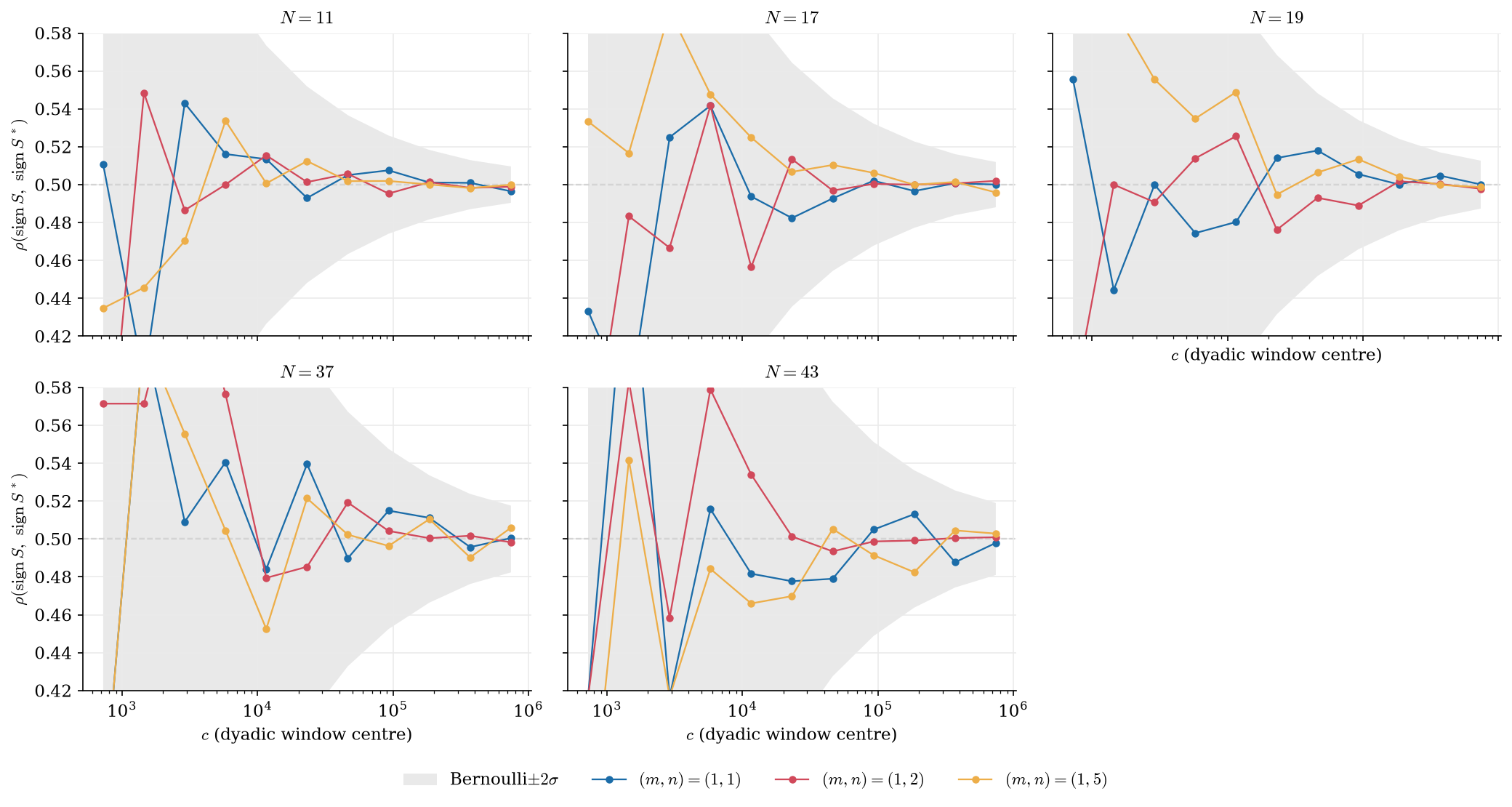}
        \caption{Correlation $\rho$ between signs of $S(m,n;c)$ and $S^*(m,n;c)$}
        \label{fig:correlation}
\end{figure}
The Pearson correlation between $S$ and $S^{*}$ 
over the same range of dyadic windows as above is also computed, and the results are shown in Figure \ref{fig:pearson_correlation}.
Note this is a correlation of the summands, 
not of the Linnik partial sums themselves.
More precisely, what is plotted is 
\[
r(S,S^*)(x) = 
\frac{\sum_{c\in W_x} (S(c)-\overline{S})(S^*(c)-\overline{S^*})}%
{\sqrt{\sum_{c\in W_x} (S(c)-\overline{S})^2}\sqrt{\sum_{c\in W_x} (S^*(c)-\overline{S^*})^2}},
\]
where $S(c)=S(m,n,c)$, $S^{*}(c)=S^{*}(m,n,c)$, $W_x=\{c; N\mid c, \;  x < c \le 2x\}$ is the dyadic window and $\overline{S}$ and $\overline{S^*}$ are the means of $S$ and $S^*$ over the window, respectively.
\begin{figure}
        \centering
        \includegraphics[width=0.8\textwidth]{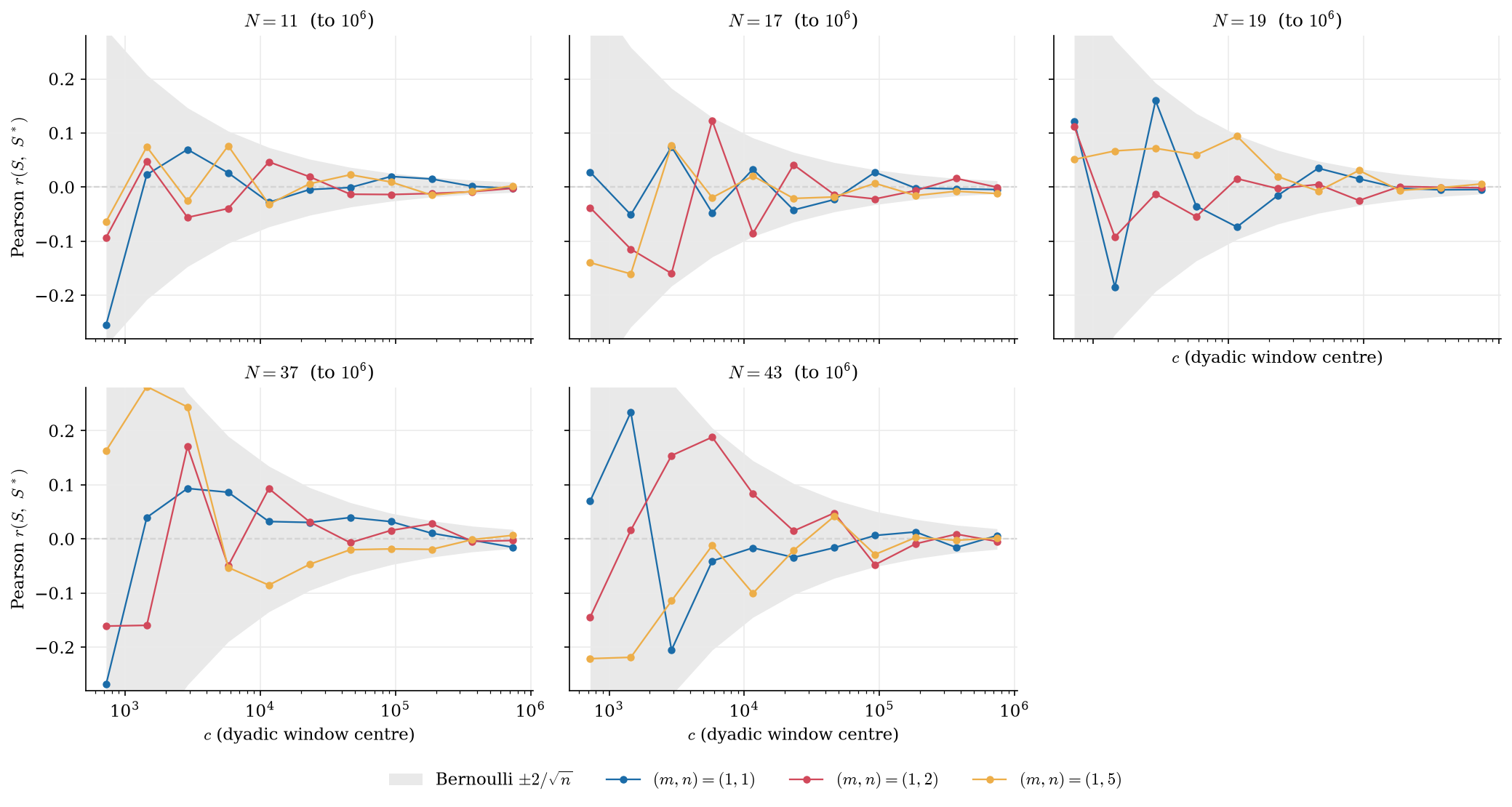}
        \caption{Pearson correlation $r(S,S^*)(x)$ for the values of $S(m,n;c)$ and $S^*(m,n;c)$ in the dyadic window $W_x$}
        \label{fig:pearson_correlation}
\end{figure}

Based on this data, we make the following conjecture concerning the lack of correlation between $S(m,n;c)$ and $S^*(m,n;c)$.
\begin{conj}\label{no-correlation}
The sums $S(m,n;c)$ and $S^*(m,n;c)$ are uncorrelated in the following senses.
\begin{enumerate}
\item For any $m,n\neq0$ with $m+n \neq 0$, $\rho_{m,n}(x)\to 1/2$ as $x\to\infty$.  That is, the signs of $S(m,n;c)$ and $S^*(m,n;c)$ are not correlated.
\item For any $m,n\neq0$ with $m+n \neq 0$, suppose that the two variances defining $r(S, S^*)(x)$
are positive for all sufficiently large $x$. Then $r(S,S^*)(x)\to 0$ as $x\to\infty$.  
\end{enumerate}
\end{conj}

\subsection{The Tauberian theorem}\label{num-Taub}

It follows from Theorem \ref{main} that for any fixed non-zero integers $n$ and $k>6$,
if the exceptional eigenvalues are absent, then \eqref{final}  becomes, as $X \to \infty$:
\[
r(X) = \mathcal{S}^*(n,k,X) - A(n) X = 
\sum_{j>M} \left ( R^{+}_j(n)X^{s_j}+R^{-}_j(n)X^{\bar s_j} \right ) + 
 \mathcal{E}^{*}(n,k, X) + O(X^{1/2-\varepsilon}),
\]
where $\mathcal{S}^*(n,k,X)$ is the smoothed $S^*$ sum, $A(n)$ is an explicit constant, $R^{\pm}_j(n)$ 
are as in Theorem \ref{main}, and $\mathcal{E}^*(n,k, X)$ is the full boundary contribution defined there.
The formulas below, using the expression of $\mathcal{E}^*(n,k, X)$ as a constant, $C$, times $X^{1/2}$, 
apply to the case when $E^*(z,s)$ is holomorphic and vanishes at $s=\frac12$. It is easy to see that this is the case 
for the levels and weights considered in the numerical experiments.
In general, $\mathcal E^*(n,k,X)/\sqrt X$ is the polynomial of degree at most two in $\log X$ specified in \eqref{E^star}.
In addition to verifying the asymptotic formula, as illustrated in Figure \ref{fig:tauberian},
we can also investigate the size and behaviour of the remainder terms.

\begin{figure}
        \centering
        \includegraphics[width=0.8\textwidth]{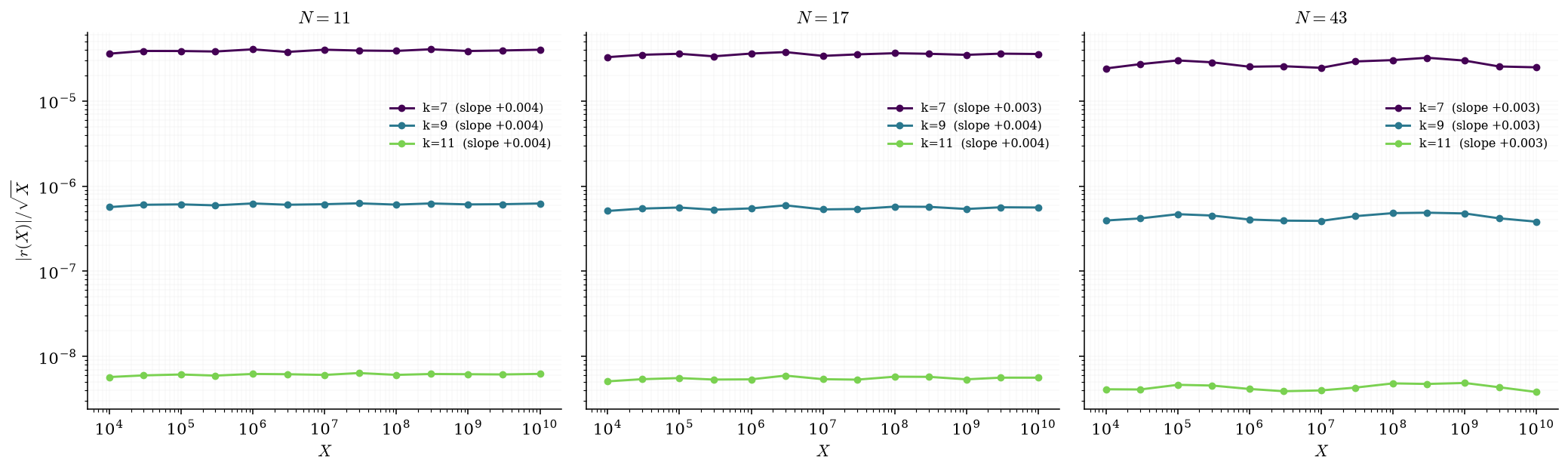}
        \caption{The normalised residual $|r(X)|/\sqrt{X}$ in log--log scale for $N=11$, $17$ and $43$, $n=1$ and $k=7,9,11$. The dashed line has slope $-1/2$; the fitted slopes are $0.004\pm0.001$, uniformly in level and weight.}
        \label{fig:tauberian}
\end{figure}

Both terms are of order $O(X^{1/2})$ but the spectral residue sum over embedded eigenvalues is oscillatory and the boundary term is monotone.
Writing $\mathcal{E}^*(n,k, X)=C\cdot X^{1/2}$ and combining the contributions from 
$s_j$ and $\bar{s}_j$ we find that 
\begin{equation}
        \label{eq:rXparams}
\frac{r(X)}{\sqrt{X}} = C + \sum_{j>M} B_j \cos(t_j \log X - \theta_j) 
  + O(X^{-\varepsilon}),
\end{equation}
where $B_j = 2|R_j^{+}(n)|$ and $\theta_j=\arg R_j^{+}(n)$. 
As an alternative to computing the actual values we can also view the 
parameters $C$, $B_j$ and $\theta_j$ as free variables and do a numerical fit.
This way we decompose $r(X)/\sqrt{X}$ as a constant, $C$, plus a slow decaying drift
corresponding to $O(X^{-\epsilon})$ and an oscillatory sum of cosines. 
It turns out that a good choice is to approximate the slow drift by a degree $3$ polynomial in $u=\log X$ 
which we obtain by a least-square fit. 
The oscillatory terms and the slow drift are illustrated in 
Figure \ref{fig:tauberian1} for $k=7$, $n=1$ and $X\leq10^{10}$.

We conclude that these numerical findings are consistent with the absence of exceptional eigenvalues since otherwise, the second term in the RHS would have order $O(X^c)$ with $c>1/2$. In theory, it could happen that exceptional eigenvalues exist but that all $R_j^+(n)$ vanish and thus, our asymptotics do not see it. However, this is unlikely because $R^+_j(n)$ are explicit scalar products of $\sum_{\substack{\ell : \lambda_\ell=\lambda_j}} L_{f \otimes \eta_{\ell}}(s_j)\rho_\ell(n)$ and thus the vanishing of  $R^+_j(n)$ would impose linear dependence relations on $L_{f \otimes \eta_{\ell}}( s_j)$ with coefficients $\rho_j(n)$ for every $n \neq 0.$

We formalize these expectations in the following Conjecture concerning sums of Ramanujan sums that involve modular symbols.

\begin{conj}\label{conj-ram-sums}
\begin{enumerate}
\item\label{part-one-ramsums} Fix $j>0$.  Then there exists an $n\in \mathbb{Z}$, $n\ne0$, such that $R_j^+(n)\neq0$.
\item\label{part-two-ramsums} For every non-zero integer $n$, every $k>6$ and every $\varepsilon>0$, $r(X)=\text{O}_{n,k,\epsilon}(X^{1/2+\varepsilon})$.
\end{enumerate}
\end{conj}

From Theorem~\ref{main}, we see that Conjecture~\ref{conj-ram-sums} implies that there are no exceptional eigenvalues.  In fact, given $n$ as in 
part~\ref{part-one-ramsums}, it 
would suffice to find a single $k$ such that the estimate in part~\ref{part-two-ramsums} holds.

\begin{figure}
        \centering
        \includegraphics[width=0.8\textwidth]{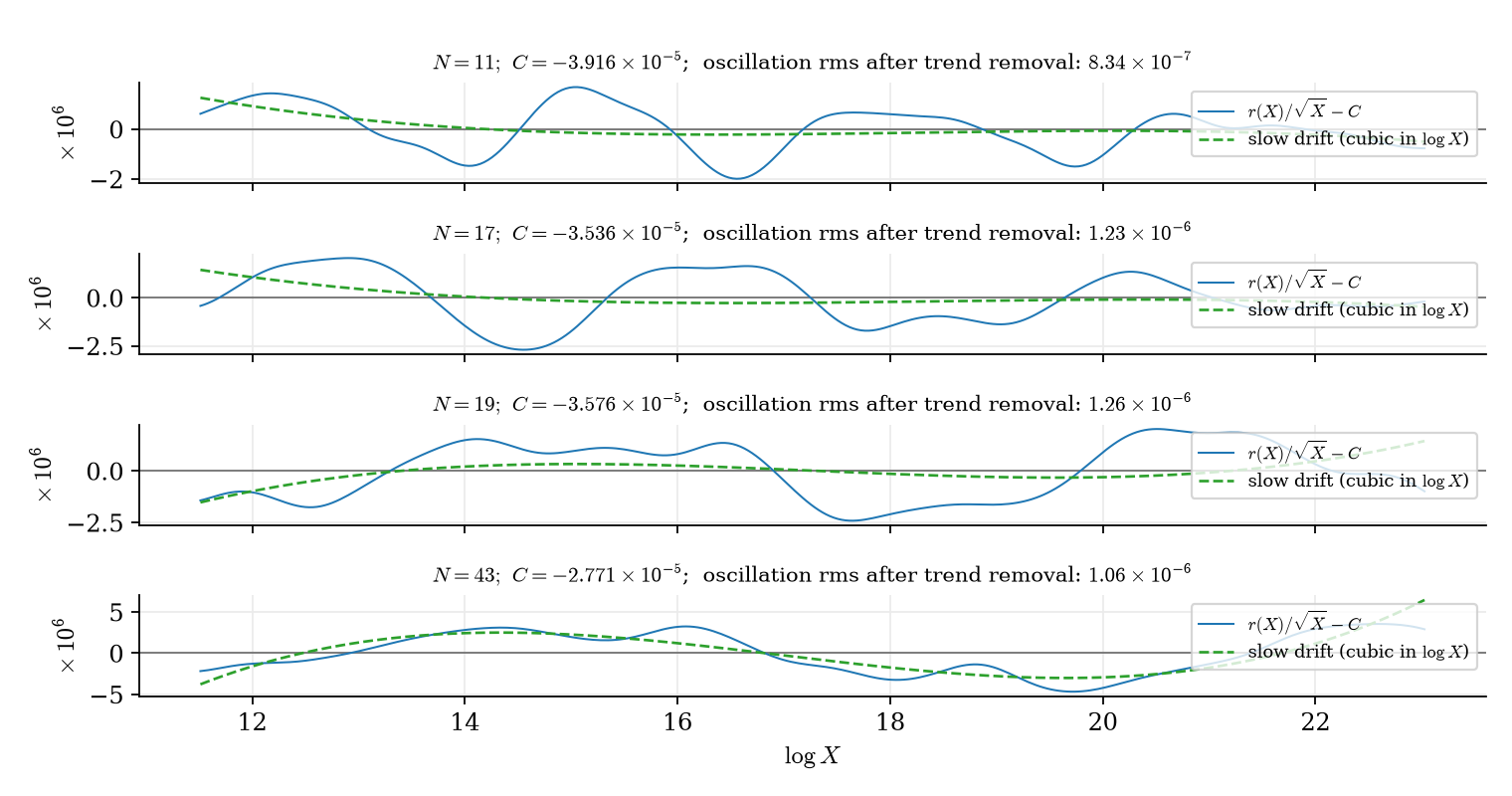}
        \caption{
        The oscillatory term: $r(X)/\sqrt{X} - C$  ($k=7$, $n=1$, $X\leq10^{10}$)
        }        
        \label{fig:tauberian1}
\end{figure}

\subsection{Eisenstein contribution}

In this section we will give numerical estimates of the boundary integral $\mathcal{E}^*(n,k,X)$ in \eqref{eq:rXparams} and compare it with the numerical fit 
obtained from $r(X)/\sqrt{X}$.
First note that by 
differentiating \eqref{eq:phi_ns}, and using that $K_\nu(y)$ is even in the order $\nu$ (so that the derivative of the Bessel factor vanishes at $s=\tfrac12$), we find that 
\[
\mathcal{E}^*(n,k,X)
   =\frac{\pi\,I_0(2\pi|n|)\,K_0(2\pi|n|)}{4\,\Gamma\!\left(k+\tfrac12\right)}\;
     \phi^{*\prime}(n,\tfrac12)\;X^{1/2},
\]
where 
$\phi^{*\prime}(n,s)$ is the derivative of the analytic 
continuation of $\phi^*(n,s)$ and if we write 
\begin{equation}\label{eq:hD}
h(s)=\frac{\pi^{s}|n|^{s-1}}{\Gamma(s)} \quad\textrm{and}\quad
D(s)=\sum_{c\ge1}(pc)^{-2s}S^*(n,0;pc)
\end{equation}
then $\phi^*(n,s)=h(s)D(s)$ and 
\begin{equation}\label{eq:phiprime}
  \phi^{*\prime}(n,\tfrac12)
   =\frac{1}{\sqrt{|n|}}\Big[\big(\log(4\pi|n|)+\gamma_E\big)\,D(\tfrac12)+D'(\tfrac12)\Big],
\end{equation}
where we stress that $D(\frac{1}{2})$ and $D'(\frac{1}{2})$ are values of the analytic continuation of $D(s)$
and its derivative, as the Dirichlet series is not absolutely convergent at this point.

To obtain a numerical estimate at the point $s=1/2$ we {\bf assume} 
that $S^{*}(n,0,pc)$ satisfies a Linnik-type bound of a form which is observed in the experimental data.
More precisely, we assume that 
with $\kappa = a(n)p^{2}/(\pi|n|\operatorname{Vol}(\Gamma\backslash\mathbb{H}))$ --- the value forced by Theorem \ref{0thmn}, and reproduced by the data to five digits,
the centred fluctuation $E(c)=S^*(n,0,pc) - 2\kappa c$ satisfies 
 \begin{equation}
 \sum_{c\le X} E(c) \ll X^{1+\epsilon}.
 \end{equation} 
It can then be shown that  $D(s)$ has a meromorphic continuation to 
$\Re(s)>1/2$ with {at most} a simple pole at $s=1$ and, we can write 
\[
D(s)= P(s) + S(s) = 2\kappa p^{-2s} \zeta(2s-1) + p^{-2s}\sum_{c\ge 1} E(c) c^{-2s},
 \]
 where the last series is convergent for $\Re(s)>1/2$. The two 
sums 
\begin{equation}\label{eq:reg}
  D(\tfrac12)=-\frac{\kappa}{p}+\sum_{c\ge1}\frac{E(c)}{pc},
  \qquad
  D'(\tfrac12)=\frac{2\kappa}{p}\log\frac{p}{2\pi}-2\sum_{c\ge1}\frac{E(c)\log(pc)}{pc}.
\end{equation}
{are understood in the Abel sense, i.e., we insert a factor
$(pc)^{-2\delta}$ into each infinite sum and let
$\delta\downarrow0$. Under the regularity assumption at
$s=\tfrac12$ made above, these limits agree with the
analytic continuation. We use Riesz sums of orders
$2$--$5$ as numerical approximations to these values.}
The results of these calculations in the case of levels $11$ and $43$ 
are shown in Table \ref{tab:riesz}.
To simplify notation we have written $S_C$ and $S'_C$ for the sums $S(s)$ and $S'(s)$ truncated at $c=C$
using either the ``naive'' sum or a Riesz sum.
Observe that naive truncations of the pole-subtracted boundary sums oscillate. 
At $p=11$ the $\log$-weighted sum ranges over
$[-4.6,+2.7]$ as the cutoff moves while the four Riesz orders agree to $\pm0.09$, a
$\sim40$-fold stabilisation. The residual monotone drift in the Riesz order is the finite-$C$ bias.
 At $p=43$ the same machinery on a $4\times$ shorter
table gives a $2.3\times$ larger relative spread on the dominant $S_C'$-sum.
\begin{table}[ht]
\centering
\small
\setlength{\tabcolsep}{3.6pt}
\begin{tabular}{l rr rr}
\toprule
& \multicolumn{2}{c}{$p=11$ \ ($c_{max}=16113$)} & \multicolumn{2}{c}{$p=43$ \ ($c_{max}=4121$)}\\
\cmidrule(lr){2-3}\cmidrule(lr){4-5}
$C=c_{\max}$ & \multicolumn{1}{c}{$S_C$} & \multicolumn{1}{c}{$S'_C$}
 & \multicolumn{1}{c}{$S_C$} & \multicolumn{1}{c}{$ S'_C$}\\
\midrule
naive at $C/4$   & $+0.4260$ & $-4.6167$ & $+0.2653$ & $-1.3364$\\
naive at $C/2$   & $+0.0930$ & $+2.7470$ & $+0.2719$ & $-1.2084$\\
naive at $3C/4$  & $+0.2505$ & $-0.8409$ & $-0.0900$ & $+7.2727$\\
naive at $C$     & $+0.3947$ & $-4.3246$ & $+0.4097$ & $-4.6542$\\
\midrule
Riesz $r=2$          & $+0.2717$ & $-1.3248$ & $+0.3296$ & $-2.6345$\\
Riesz $r=3$          & $+0.2753$ & $-1.4066$ & $+0.3439$ & $-3.0162$\\
Riesz $r=4$          & $+0.2777$ & $-1.4624$ & $+0.3424$ & $-3.0242$\\
Riesz $r=5$          & $+0.2791$ & $-1.4965$ & $+0.3345$ & $-2.8797$\\
\midrule
$P(1/2),\ P'(1/2)$ & $-0.2786$ & $+0.3121$ & $-0.2971$ & $+1.1427$\\
\midrule
$D(\tfrac12),\ D'(\tfrac12)$
  & $-0.0027$ & $-1.111$ & $+0.0405$ & $-1.75$\\
est. error & $\pm 0.0037$ & $\pm 0.086$
& $\pm 0.0072$ & $\pm 0.19$\\
  \bottomrule
\end{tabular}
\caption{The Riesz evaluation of $D(1/2)$ and $D'(1/2)$ compared 
to naive truncation.}
\label{tab:riesz}
\end{table}
Table \ref{tab:eis} shows the comparison of the estimated Eisenstein boundary term $\mathcal{E}(n,k,X)$ with the measured non-oscillating error coefficient $C$ in \eqref{eq:rXparams} for various prime levels $p$.
The agreement degrades with $p$ because we worked with a fixed $X_{\max}$ so that $c_{\max}\propto 1/p$ rather than a fixed number of moduli, and the {Riesz approximation to $D'(\tfrac12)$}
is progressively under-resolved.
\begin{table}[ht]
\centering
\begin{tabular}{rrccc}
\toprule
$p$ & $c_{\max}$ & estimated $\mathcal{E}(n,k,X)/\sqrt{X}$ & measured $C$ & ratio\\
\midrule
$11$ & $16113$ & $-3.75\times10^{-5}\ (\pm3.3\cdot10^{-6})$ & $-3.88\times10^{-5}\ (\pm1.2\cdot10^{-6})$ & $0.97$\\
$17$ & $10426$ & $-3.31\times10^{-5}\ (\pm2.8\cdot10^{-6})$ & $-3.52\times10^{-5}\ (\pm1.2\cdot10^{-6})$ & $0.94$\\
$19$ & $9328$  & $-2.96\times10^{-5}\ (\pm1.6\cdot10^{-5})$ & $-3.54\times10^{-5}\ (\pm1.8\cdot10^{-6})$ & $0.83$\\
$37$ & $4790$  & $-3.49\times10^{-5}\ (\pm1.9\cdot10^{-5})$ & $-2.98\times10^{-5}\ (\pm1.9\cdot10^{-6})$ & $1.17$\\
$43$ & $4121$  & $-5.43\times10^{-5}\ (\pm7.3\cdot10^{-6})$ & $-2.74\times10^{-5}\ (\pm2.5\cdot10^{-6})$ & $1.98$\\
\bottomrule
\end{tabular}
\caption{The estimated Eisenstein boundary term versus the measured non-oscillating error coefficient, $n=1$, $k=7$. }
\label{tab:eis}
\end{table}

A further remark is in order. First, at the levels $p=37,43$ one has $f|_{2}W_p=f$ and 
$L_f(1)=0$, 
so both terms of the mixed scattering coefficient $\phi^*_{\infty 0}$ vanish 
and the central value $\phi^*(1,\tfrac12)$, hence $D(\tfrac12)$, vanishes \emph{exactly}.
 This provides an absolute benchmark for the Riesz errors: at $p=37$ 
 the Riesz evaluation gave $D(\tfrac12)=+0.002\pm 0.022$, which is consistent,
 while for $p=43$ we see in Table \ref{tab:riesz} that $D(\tfrac12)=+0.0405\pm 0.0072$,
showing the quoted error to be optimistic, consistent with the ratio in
  Table \ref{tab:eis}. 

\subsection{Maass cusp form contribution}
\label{Maasscontr}
For simplicity, consider first the case where we know that there are no exceptional eigenvalues. Then we are again
in the situation of \eqref{eq:rXparams} where each term in the Maass sum has three free parameters,
the amplitude, $B_j$, the phase, $\theta_j$, and the frequency, $t_j$.
Note that if there are exceptional eigenvalues, then this would be visible already at the level of fitting the Eisenstein contribution
or the slow drift term. For more precise statements see the paragraphs at the end of this section.

Write the amplitude from \eqref{eq:rXparams} as
\begin{equation}
R_j^+(n)=B_{k,n}(t_j)\cdot\sum_{\ell}L_{f\otimes\eta_\ell}(s_j)\rho_\ell(n).
\end{equation}
{From the proof of Theorem~\ref{main} we can deduce that
$R_j^+(n)\ll_{f,p,n,k,\varepsilon}
(1+|t_j|)^{1/2-k+\varepsilon}$.
Low-frequency dominance is observed numerically.}

The frequency resolution, meaning the precision with which we can distinguish between different spectral parameters, 
is given by $\delta t\sim \frac{2\pi}{\Delta \log X}$. The numerical precision obtainable by frequency 
analysis of $r(X)/\sqrt{X}$ is therefore limited by the range of $X$ we have access to.
For instance, with $10^4 \le X\le 10^{10}$ we have $\Delta \log X \sim 13.81$ and hence $\delta t \sim 0.45$.
To achieve a resolution of, say, $\delta t = 0.1$, on $\Gamma_0(p)$, it is necessary to compute $r(X)$ 
 with $X_{\max} \ge 10^{27}$, hence to evaluate Kloosterman sums for $c$ up to $\sqrt{\pi X_{\max}}/p \approx 5.6\cdot 10^{13}/p$, 
 which is currently not feasible. 
 All our computations have been performed on a server with 32 Intel Xeon Silver 4514Y cores and 500GB of RAM.
 The largest evaluation of $r(X)$ used $c_{\max} \sim 500,000$ for $N=11$ 
 and this took about 16.1h wall-clock time using all 32 cores. 
 Ignoring potential issues with double floating-point precision under- and overflows 
computing up to $c_{\max}\approx 5\cdot 10^{12}$ (as required at $p=11$) on the same hardware using the same 
method would take roughly $3.7\times 10^{11}$ years.

As an example of the resolution we can achieve, we consider the case of $n=1$ and $k=7$ for the prime levels $N=11, 17, 19$ and $43$ and 
fit a single cosine to the dominant frequency over a window of $10^4 \le X\le 10^{10}$. See Table \ref{tab:spec}.
\begin{table}[ht]
\centering
\begin{tabular}{rcccl}
\toprule
$N$ & extracted $t$ & osc.\ RMS & var.\ expl. & nearest LMFDB $t_j$\\
\midrule
$11$ & $2.20$ & $9.1\cdot10^{-7}$ & $87\%$ & $2.033,\ 2.484$\\
$17$ & $1.88$ & $1.9\cdot10^{-6}$ & $91\%$ & $1.441, 1.850,\ 1.968,\ 1.979$\\
$19$ & $2.54$ & $7.7\cdot10^{-7}$ & $59\%$ & $1.092,\ 1.330,\ 1.870,\ 2.220,\ 2.297,\ 2.610$\\
$43$ & $2.60$ & $1.1\cdot10^{-6}$ & $68\%$ & $0.655,\ \ldots \textrm{[10 eigenvalues]}\ldots, 2.273,\ 2.295,\ 2.325$\\
\bottomrule
\end{tabular}
\caption{Dominant oscillation frequency of the detrended residual, per level ($n=1$, $k=7$)}
\label{tab:spec}
\end{table}
As an alternative to fitting a single cosine and trying to extract the dominant frequency, we can also
use the known certified eigenvalues of Maass cusp forms from the LMFDB \cite{LMFDB} and fit only the amplitudes and phases. 
The predicted amplitudes are computed from the explicit formula: writing $R_j^+(n)={-}B_{k,n}(t_j)\,\rho_j(1)^2\,L_{\mathrm{an}}(s_j)\lambda_j(n)$, where $\lambda_j(m)$ are the Hecke eigenvalues of $\eta_j$ and $L_{\mathrm{an}}(s)=\sum_m\lambda_f(m)\lambda_j(m)m^{-s}$ with $\lambda_f(m)=a(m)/\sqrt m$, the Rankin--Selberg value is continued to $s_j$ through the completed degree-four $L$-function, whose local factor at $p$ (of Steinberg$\times$Steinberg type) and functional-equation sign were determined empirically by a consistency test on the coefficient data; the Petersson normalisation $\rho_j(1)^2$ follows from Rankin--Selberg unfolding, with $\operatorname{Res}_{s=1}\sum_m\lambda_j(m)^2m^{-s}$ measured directly from the certified LMFDB coefficients. No parameter is fitted. Using a window $X\le 10^{13}$ ($\delta t\approx 0.34$) we obtain the results in Table \ref{tab:deepwin}.
\begin{table}[ht]
\centering
\footnotesize
\begin{tabular}{rlcccl}
\toprule
$p$ & frequency (LMFDB) & fitted amplitude & comp.\ $2|R_j^+(1)|$ & comp/fit\\
\midrule
$11$ & $t_1=2.0331$ & $(9.7$--$10.0)\cdot10^{-7}$ & $9.81\cdot10^{-7}$ & $0.98$--$1.01$ \\
$11$ & $t_1$, $k{=}11$ & $1.40\cdot10^{-10}$ & $1.39\cdot10^{-10}$ & $1.01$ \\
$17$ & $t_1=1.4414$ & $(1.12$--$1.14)\cdot10^{-6}$ & $1.12\cdot10^{-6}$ & $0.98$--$1.00$ \\
$19$ & $t_1=1.0920$ & $(2.22$--$2.23)\cdot10^{-6}$ & $2.23\cdot10^{-6}$ & $\mathbf{1.00}$ \\
$43$ & $t_1=0.6555$ & $(1.7$--$2.1)\cdot10^{-6}$ & $2.35\cdot10^{-6}$ & $1.1$--$1.4$  \\
\bottomrule
\end{tabular}
\caption{
        Fitted and computed per-eigenvalue amplitudes ($n=1$, $k=7$ except where noted $k=11$). }
\label{tab:deepwin}
\end{table}

Conversely, releasing a frequency in the fit \emph{measures} the corresponding eigenvalue 
from the Kloosterman-sum data alone: at level $11$ this gives 
$\hat t_1 = 2.0280\pm 0.0029$, against the certified value $2.03309\ldots$. 
The full decomposition of $r(X)/\sqrt{X}$ into its Eisenstein and Maass layers, together with the computed first-eigenvalue amplitudes $\pm 2|R_1^+(1)|$, is displayed in Figure \ref{fig:decomp}.

\begin{figure}
        \centering
        \includegraphics[width=0.9\textwidth]{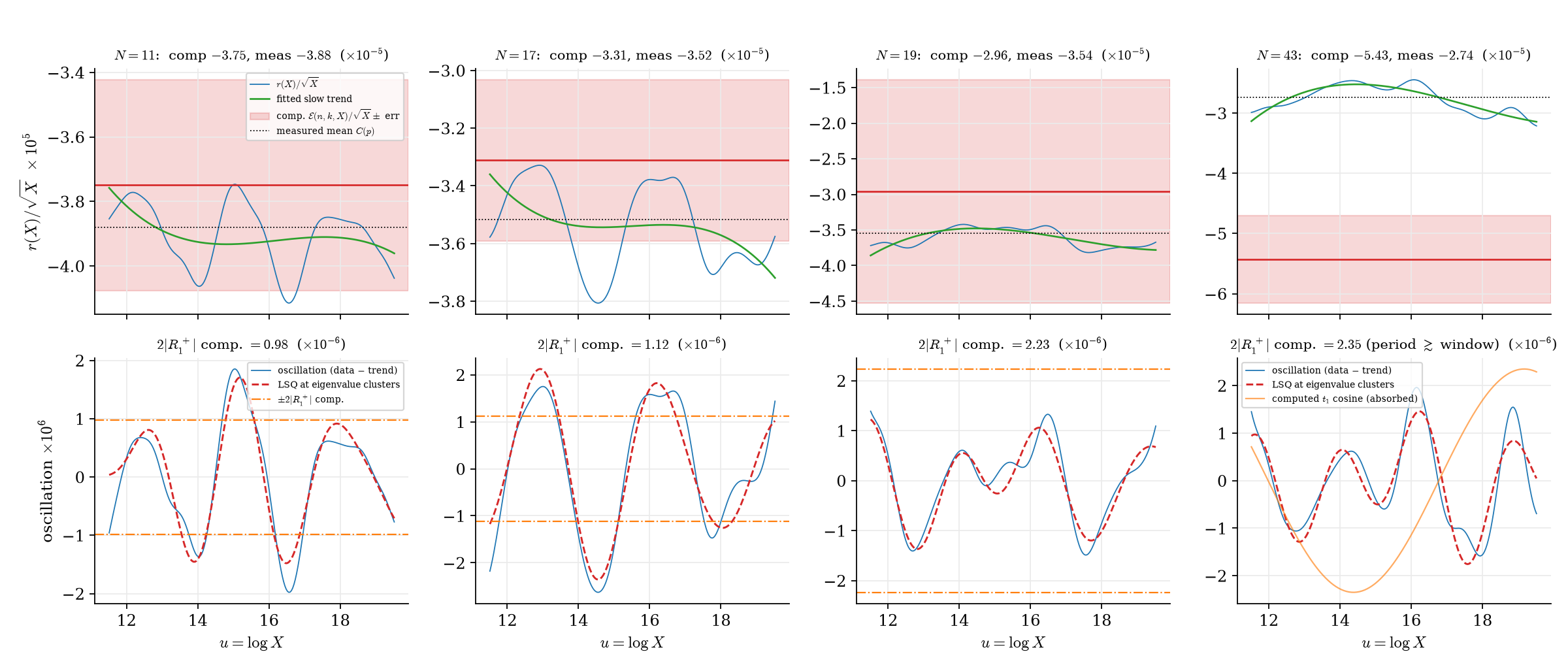}
        \caption{The two remainder terms at $k=7$, $n=1$. Top (Eisenstein layer): $r(X)/\sqrt{X}$ with its slow trend; the shaded band is the regularised prediction of the boundary term with its error. 
        Bottom (Maass layer): the trend-subtracted oscillation with the least-squares cosine model at the 
        certified frequencies from LMFDB; dash-dotted lines mark the computed $\pm2|R_1^+(1)|$, and at $N=43$ the predicted slow $t_1=0.655$ cosine, whose period is comparable to the window, is drawn explicitly.}
        \label{fig:decomp}
\end{figure}

Finally, these data bear on the exceptional spectrum from a second direction,
complementing the block-convergence test of the first subsection above: an exceptional
eigenvalue $s=\tfrac12+\delta$ would contribute a term growing like $e^{\delta\log X}$ to
the normalised residual, and profiling such a column against the deep grids bounds the
corresponding residue at, e.g., the Kim--Sarnak point $\delta=7/64$ by $1.5\cdot10^{-6}$
($p=11$). {This is} three orders of magnitude below its natural size with the caveats that
the bound constrains the product with $L_{f\otimes\eta}(s)$ and degenerates as
$\delta\to0$. Since the cost of the underlying tables scales like $X\log X/p$,
\emph{inversely} in the level, the same experiment is inexpensive at prime levels beyond
the range of current rigorous verifications of Selberg's conjecture. As a first instance
we have carried it out at $N=997$ (about 20 minutes of computation on 32 cores for the table to
$X=10^{13}$): the ratios of the main terms of Theorems \ref{main} and \ref{0thmn} are confirmed at
this level $1$ with errors $\pm 4\cdot 10^{-6}$ and $\pm 0.001$, respectively, 
the residual shows no growing component, and a natural-size exceptional
residue is excluded for $\delta\gtrsim0.2$, though at the Kim--Sarnak point
$\delta=7/64$ only marginally. The natural residue scale itself shrinks like
$1/\mathrm{Vol}(\Gamma_0(N)\backslash\mathbb{H})$, so a sharper statement at large
level requires $X\gtrsim10^{15}$, which remains entirely practical there.

The numerical evidence and algorithms used in this section are presented in more detail
in the notes \cite{StrZ} by the third author.

\end{document}